\documentclass[10pt,a4paper]{article}

\usepackage[utf8]{inputenc}

\usepackage{amsthm}
\usepackage{amssymb}
\usepackage[tbtags]{amsmath}
\usepackage{mathtools}
\usepackage{a4wide}
\usepackage{upgreek}
\usepackage{amsfonts}
\usepackage{cite}
\usepackage{enumitem}
\usepackage{hyperref}
\usepackage{dsfont}
\usepackage{authblk}
\usepackage{xcolor}

\newcommand{\Lip}{\operatorname{Lip}}

\renewcommand{\div}{\operatorname{div}}
\newcommand{\law}{\operatorname{law}}

\DeclareMathOperator*{\esssup}{ess\,sup}

\newcommand{\hook}{\hookrightarrow}
\newcommand{\harp}{\rightharpoonup}
\newcommand{\harpstar}{\overset{\ast}{\harp}}
\newcommand{\grad}{\nabla}
\newcommand{\laplace}{\Delta}
\newcommand{\Borel}{\operatorname{Borel}}

\newcommand{\Ac}{\mathcal{A}}
\newcommand{\Apr}{A_{\rm{pr}}}
\newcommand{\Eb}{\mathbb{E}}
\newcommand{\Fc}{\mathcal{F}}

\newcommand{\Fb}{\mathbb{F}}
\newcommand{\Gc}{\mathcal{G}}

\newcommand{\Hbar}{\bar{H}}

\newcommand{\Jc}{\mathcal{J}}

\newcommand{\Mc}{\mathcal{M}}
\newcommand{\Nb}{\mathbb{N}}
\newcommand{\Pb}{\mathbb{P}}
\newcommand{\Rb}{\mathbb{R}}

\newcommand{\Rc}{\mathcal{R}}
\renewcommand{\Re}{R^\varepsilon}
\newcommand{\Tc}{\mathcal{T}}
\newcommand{\Uc}{\mathcal{U}}
\newcommand{\Ut}{\tilde{U}}
\newcommand{\Uh}{\hat{U}}
\newcommand{\Ue}{U^\varepsilon}
\newcommand{\Uhe}{U_{h_\varepsilon}}
\newcommand{\upe}{\upnu^\varepsilon}
\newcommand{\upeb}{\bar{\upnu}^\varepsilon}
\newcommand{\upet}{\tilde{\upnu}^\varepsilon}
\newcommand{\Upe}{\Upsilon^\varepsilon}

\newcommand{\Vc}{\mathcal{V}}
\newcommand{\vb}{\bar{v}}
\newcommand{\vt}{\tilde{v}}

\newcommand{\Ye}{Y^\varepsilon}
\newcommand{\Ho}{\overline{H}}
\newcommand{\Vo}{\overline{V}}

\providecommand{\keywords}[1]
{
	{\small
  	\textbf{\textit{Keywords---}} #1
  	}
}

\providecommand{\msc}[1]
{
	{\small
  	\textbf{\textit{MSC2010---}} #1
  	}
}

\newtheorem{theorem}{Theorem}[section]
\newtheorem{lemma}[theorem]{Lemma}

\newtheorem{proposition}[theorem]{Proposition}
\newtheorem{definition}[theorem]{Definition}

\theoremstyle{definition}

\theoremstyle{remark}

\numberwithin{equation}{section}
\title{Large and moderate deviations principles and central limit theorem for the stochastic 3D primitive equations with gradient dependent noise}
\author{Jakub Slav\'{i}k\footnote{email: \href{mailto:slavik@utia.cas.cz}{slavik@utia.cas.cz}, ORCID: \href{https://orcid.org/0000-0001-6465-663X}{0000-0001-6465-663X}}}
\affil{The Czech Academy of Sciences, Institute of Information Theory and Automation\protect\\Pod~Vod\'{a}renskou v\v{e}\v{z}\'{i} 4, 182 00 Prague 8, Czech Republic}
\date{}

\begin{document}

\maketitle

\begin{abstract}
We establish the large deviations principle (LDP) and the moderate deviations principle (MDP) and an almost sure version of the central limit theorem (CLT) for the stochastic 3D viscous primitive equations driven by a multiplicative white noise allowing dependence on spatial gradient of solutions with initial data in $H^2$. The LDP is established using the weak convergence approach of Budjihara and Dupuis and uniform version of the stochastic Gronwall lemma. The result corrects a minor technical issue in  Z.~Dong, J.~Zhai, and R.~Zhang: Large deviations principles for 3D stochastic primitive equations, J.~Differential Equations, 263(5):3110–3146, 2017, and establishes the result for a more general noise. The MDP is established using a similar argument.
\end{abstract}


\msc{60H15, 60F10, 35Q86}

\tableofcontents

\section{Introduction}

The primitive equations are one of the fundamental models in geophysical fluid dynamics, see e.g.\ \cite{pedlosky,vallis} and the references therein. They can be derived from the Navier-Stokes equations using the Boussinesq approximation and the hydrostatic balance. The main aim of this paper is to study limiting small noise behaviour of solutions of the stochastic 3D primitive equations driven by a multiplicative white noise. In particular, we establish the large deviations principle, the moderate deviations principle and an almost sure version of the central limit theorem.

The systematic study of the deterministic primitive equations began in the 1990s in a series of papers \cite{lions1992b,lions1992a,lions1993}. Local existence of strong solutions has been established in \cite{guillengonzalez2001,hu2003}, global existence of strong solutions and uniqueness in $L^2$ has been shown in \cite{cao2007,kobelkov2007,kukavica2007} using different methods. More recent results include well-posedness in $L^p$ spaces \cite{hieber2016} and well-posedness of primitive equations with partial viscosity and/or diffusivity, see \cite{li2018,hussein2019} and the references therein.

In the stochastic setting, the 3D primitive equations driven by additive noise have been studied in \cite{guo2009}. Local and global existence of strong solutions (both in the stochastic and PDE sense) has been established in \cite{debussche2011,debussche2012}. Existence and regularity of invariant measures have been shown in \cite{glatt-holtz2014}. Time discretization of weak solutions (both in stochastic and PDE sense) has been treated in \cite{glatt-holtz2017}. In this paper, we use the existence theorem from \cite{brzezniakslavik} which allows the noise term $\sigma$ to depend on the spatial gradient of the velocity and temperature. Such noise term are physically reasonable, see e.g.\ \cite{mikulevicius2001,mikulevicius2004} and the references therein.

The stochastic 3D primitive equations driven by a scaled multiplicative noise can be written in an abstract form as 
\begin{equation}
	\label{eq:intro.pe.eps}
	d\Ue + \left[ A\Ue + B\left( \Ue \right) + \Apr \Ue + E\Ue + F_U \right] \, dt = \sqrt{\varepsilon} \sigma\left( \Ue \right) \, dW, \quad \Ue(0) = u_0,
\end{equation}
with $\Ue = (v^\varepsilon, T^\varepsilon)$, where $v^\varepsilon$ is the horizontal velocity and $T^\varepsilon$ is the temperature. A full form of equation \eqref{eq:intro.pe.eps} can be found in Section \ref{sect:prelim} together with the definitions of the operators and function spaces used here. The notation and the full form of the equations can be found in Section \ref{sect:prelim}. We will be studying the convergence of $\Ue$ to the solution of the deterministic equation
\begin{equation}
	\label{eq:intro.pe.zero}
	dU^0 + \left[ AU^0 + B\left( U^0 \right) + \Apr U^0 + EU^0 + F_U \right] \, dt = 0, \qquad U^0(0) = u_0,
\end{equation}
in various scales. Let $\lambda(\varepsilon)$ be a certain deviations scale and let
\[
	\Re = \frac{U^\varepsilon - U^0}{\sqrt{\varepsilon} \lambda(\varepsilon)}.
\]

If $\lambda(\varepsilon) = 1/\sqrt{\varepsilon}$, the asymptotic behaviour of $\Re$ as $\varepsilon \to 0+$ is known as the large deviations principle (LDP), see \cite{varadhan,freidlin,liu2010,budhiraja2008} and the references therein. For stochastic primitive equations, the LDP has established in \cite{gao2012} in 2D and in \cite{dong2017} in 3D. In the last mentioned reference the LDP result is obtained in the setting of \cite{debussche2012}, in particular with noise term that does not allow dependence on gradients. Also, the proof of the main result in \cite{dong2017} has certain minor technical issues which we address in Section \ref{sect:ldp.proof}. Some of related results include the LDP for the stochastic 2D Navier-Stokes equations \cite{chueshov2010} and the LDP for stochastic 2D quasi-geostrophic equations \cite{liu2013}. 

Let us formulate the first main results of this paper. Let $t > 0$ be fixed. Let $(\Omega, \Fc, \Fb, \Pb)$, $\Fb = (\Fc_t)_{t \geq 0}$, be a fixed stochastic basis satisfying the usual conditions and let $W$ be a cylindrical $\Fb$-Wiener process. The following theorem will be proved in Section \ref{sect:ldp.proof}.

\begin{theorem}[LDP]
\label{thm:ldp}
Let $\sigma$ satisfy the assumptions (\ref{eq:sigma.bnd.H}-\ref{eq:sigma.vz}) and let $F_U$ satisfy \eqref{eq:f.source}. Moreover, let $\sigma$ be such that the solution of the equation
\begin{equation}
	\label{eq:ldp.skeleton.regularity}
	dU_h + \left[ AU_h + B(U_h) + F(U_h) \right] \, dt = \sigma(U_h) h \, dt, \qquad U_h(0) = u_0,
\end{equation}
satisfies $U_H \in L^\infty(0, t; H^2)$ for all $t>0$, $u_0 \in V \cap H^2$ and $h \in L^2(0, t; \Uc)$. Then there exists $\varepsilon_0 > 0$ such that for all $u_0 \in V$ the solutions $\lbrace U^\varepsilon \rbrace_{\varepsilon \in (0, \varepsilon_0]}$ of \eqref{eq:intro.pe.eps} satisfy the LDP with a good rate function $I$ given by \eqref{eq:ldp.rate.actual}.
\end{theorem}

The additional regularity of solutions of the skeleton equation \eqref{eq:ldp.skeleton.regularity} will be used in a compactness argument, see the proof of Proposition \ref{prop:skeleton.convergence}. Additional assumptions on $\sigma$ that guarantee the desired regularity of solutions of the skeleton equation will be established in Section \ref{sect:ldp.skeleton}.

If $\lambda(\varepsilon) \to \infty$ and $\sqrt{\varepsilon}\lambda(\varepsilon) \to 0$, the asymptotic behaviour of $\Re$ as $\varepsilon \to 0+$ is known as the moderate deviations principle (MDP). The MDP for the stochastic 2D Navier-Stokes equations has been established in \cite{wang2015}. Recently, the MDP for weak solutions (in the PDE sense) of the stochastic 2D primitive equations has been proved in \cite{zhang2019}. The difficulty in obtaining the same result in 3D and for strong solutions lies in more delicate estimates of the non-linear part for which we also require higher regularity of the solution of the deterministic equation \eqref{eq:intro.pe.zero}. The following theorem will be proved in Section \ref{sect:mdp.proof}.

\begin{theorem}[MDP]
\label{thm:mdp}
Let $\sigma$ satisfy the assumptions (\ref{eq:sigma.bnd.H}-\ref{eq:sigma.vz}) and let $F_U$ satisfy \eqref{eq:f.source}. Then for all $u_0 \in V \cap H^2$ there exists $\varepsilon_0 > 0$ such that the solutions $\lbrace U^\varepsilon \rbrace_{\varepsilon \in (0, \varepsilon_0]}$ of the stochastic 3D periodic equations \eqref{eq:intro.pe.eps} satisfy the MDP with a good rate function given by \eqref{eq:mdp.rate}.
\end{theorem}

For MDP, the additional regularity of solutions of \eqref{eq:ldp.skeleton.regularity} is not needed for the compactness argument, see Section \ref{sect:mdp.skeleton}.

If $\lambda(\varepsilon) \equiv 1$, the limiting process corresponds to the central limit theorem (CLT). The CLT for weak solutions of the stochastic 2D primitive equations has been recently established in \cite{zhang2019}. We prove a weaker version with convergence only $\Pb$-a.s.\ instead of convergence in $L^2(\Omega)$, the reason for this limitation being that our definition of the solution does not guarantee that our solutions considered over a time interval $[0, t]$ form an $L^2$-integrable random variable for any $t > 0$ deterministic. The following theorem will be proved in Section \ref{sect:clt}.

\begin{theorem}[almost sure CLT]
\label{thm:clt}
Let $\sigma$, $F$ and $u_0$ be as in Theorem \ref{thm:mdp}. Let $\Uh$ be the solution of
\begin{equation*}
	d\Uh + \left[ A\Uh + B(U^0, \Uh ) + B( \Uh, U^0 ) + F_d( \Uh ) \right] \ ds = \sigma( \Uh ) \, dW, \quad \Uh(0) = u_0.
\end{equation*}
Then 
\begin{equation*}
	\frac{U^\varepsilon - U^0}{\sqrt{\varepsilon}} \to \Uh \ \text{as $\varepsilon \to 0+$ in} \ C\left( [0, t], V \right) \cap L^2\left( 0, t; D(A) \right) \ \Pb\text{-a.s.}
\end{equation*}
\end{theorem}

We assume the Neumann boundary condition on the top and bottom parts of the boundary. On the lateral part we assume periodic boundary condition similarly as in e.g.\ \cite{caolititi2014}. The choice of the lateral boundary condition rests on the required additional regularity of the solution $U^0$, i.e.\ the solution of a deterministic equation \eqref{eq:intro.pe.zero}, see Section \ref{sect:mdp.preliminary}. To the best of our knowledge, the only result providing sufficient regularity for the estimates in Section \ref{sect:mdp.preliminary} has been established in \cite{caolititi2014} for the case of periodic boundary conditions on the lateral part of the boundary and is not known for the Dirichlet boundary conditions. We emphasize that we do not require any additional regularity of the solutions of the stochastic 3D primitive equations \eqref{eq:intro.pe.eps}.

The paper is organized as follows: In Section 2 we present the functional setting and recall the sufficient condition for the LDP by Budhiraja and Dupuis \cite{budhiraja2000}. The LDP, MDP and CLT are then established in Sections 3, 4 and 5, respectively. Since the proofs of Sections 3 and 4 are quite similar, we choose to present the crucial stochastic part of the proof of the LDP in full detail in Section 3 and discuss the MDP counterpart in Section 4 only briefly. Similarly, since the estimates in Section 4 are more involved than the estimates in Section 3, we include full proofs only of the estimates in Section 4.

\section{Preliminaries}
\label{sect:prelim}

Let $L, h > 0$ and let $\Mc_0 = (0, L) \times (0, L) \subseteq \Rb^2$, $\Mc = \Mc_0 \times (-h, 0) \subseteq \Rb^3$. We decompose the boundary of $\Mc$ into
\begin{equation*}
	\Gamma_i = \overline{\Mc_0} \times \lbrace 0 \rbrace, \quad \Gamma_l = \partial \Mc_0 \times (-h, 0), \qquad \Gamma_b = \overline{\Mc_0} \times \lbrace -h \rbrace.
\end{equation*}

\subsection{Functional setting}

The reformulated\footnote{For the original system and the reformulation procedure see e.g.\ \cite[Section 2.1]{petcu2009}.} stochastic 3D primitive equations\footnote{An equation for salinity is often included as well. However, since it does not introduce any additional mathematical difficulties, it is omitted here.} are given by
\begin{gather}
	\label{eq:pe.v.reform}
	\begin{multlined}
		\partial_t v + \left(v\cdot\grad\right)v + w(v)\partial_z v + \tfrac{1}{\rho_0} \grad p_S - \beta_T g \grad \int_z^0 T \, dz' + f \vec{k} \times v\\
		- \mu_v \laplace v - \nu_v \partial_{zz} v = F_v + \sigma_1\left(v, \grad_3 v, T, \grad_3 T\right) \dot{W}_1,
	\end{multlined}\\
	\div \int_{-h}^0 v(x, y, z') \, dz' = 0,\\
	\label{eq:pe.T.reform}
	\partial_t T + \left(v\cdot\grad\right)T + w\partial_z T - \mu_T \laplace T - \nu_T \partial_{zz} T = F_T + \sigma_2\left(v, \grad_3 v, T, \grad_3 T\right) \dot{W}_2,
\end{gather}
where
\begin{gather}
	\label{eq:pe.w.definition}
	w(v)(x, y, z) = - \int_{-h}^z \div v(x, y, z') \, dz'
\end{gather}
is the vertical velocity, $v = (v_1, v_2)$ denotes the horizontal velocity, $p_S$ is the surface pressure, $f$ is the Coriolis parameter, $\mu_v$, $\nu_v$ and $\mu_T$ and $\nu_T$ are the horizontal and vertical viscosity and diffusivity coefficients, respectively. The equations are being driven by deterministic non-autonomous forces $F_v$, $F_T$ and stochastic terms $\sigma_1$ and $\sigma_2$ with multiplicative white noise in time. In the whole manuscript the symbols $\div$, $\grad$ and $\laplace$ denote the two-dimensional (that is w.r.t.\ the horizontal coordinates $x$ and $y$) divergence, gradient and Laplacian, respectively. Their three-dimensional variants will be denoted by $\div_3$, $\grad_3$ and $\laplace_3$. The equations (\ref{eq:pe.v.reform}-\ref{eq:pe.T.reform}) are supplied with the initial data
\[
	v(0) = v_0, \qquad T(0) = T_0,
\]
and the boundary conditions
\begin{align*}
	\text{on $\Gamma_i$}:& & &\partial_z v = 0, \quad \partial_z T + \alpha  T = 0,\\
	\text{on $\Gamma_l$}:& & &\text{$v$ and $T$ are periodic},\\
	\text{on $\Gamma_b$}:& & &\partial_z v = 0, \quad \partial_z T = 0.
\end{align*}

Unless specified otherwise, all the function spaces are tacitly considered to contain functions with domain $\Mc$. The Lebesgue space $L^p(\Mc)$, $p \in [1, \infty]$, will be often denoted by $L^p$. The norms on $L^p(\Mc)$ and $L^p(\Mc_0)$ may both be denoted by $\vert \cdot \vert_{L^p}$. The norm and the inner product on $L^2$ will be denoted by $\vert \cdot \vert$ and $( \cdot, \cdot)$, respectively. We will also often not specify the range of the function spaces as it should be clear from the context. Therefore, assuming $k \in \Nb$ and $p \in [1, \infty]$, both the Sobolev spaces $W^{k, p}(\Mc)$ and $W^{k, p}(\Mc; \Rb^2)$ might be denoted by $W^{k, p}$. We will also use the notation $W^{k, 2} = H^k$, $k \in \Nb$. The norm on $H^1$ will be denoted by $\Vert \cdot \Vert$.

Let $H_1$ and $H_2$ be the spaces defined by
\[
	H_1 = \left\lbrace v \in L^2\left( \Mc; \Rb^2 \right) \mid \div \int_{-h}^0 v \, dz = 0, \int_{-h}^0 v \, dz \ \text{is periodic in $\Mc_0$} \right\rbrace, \quad H_2 = L^2\left( \Mc \right), 
\]
and let $H = H_1 \times H_2$. Equipped with the inner products of $L^2$, the spaces $H_1$, $H_2$ and $H$ are Hilbert spaces. Let $P_{H_1}: L^2 \to H_1$ be the hydrostatic Helmholtz-Leray projection, see \cite[Lemma 2.2]{lions1992b} and \cite[Proposition 4.3]{hieber2016}, and let
\[
	P_H = \begin{pmatrix}
		P_{H_1}\\
		I
	\end{pmatrix},
\]
where $I$ is the identity on $L^2$, denote the hydrostatic Helmholtz-Leray projection on the space $H$. Let $V_1$ be the space defined by
\[
	V_1 = \left\lbrace v \in H^1\left(\Mc; \Rb^2\right) \mid \div \int_{-h}^0 v \, dz = 0, v \ \text{is periodic w.r.t.\ $x$ and $y$} \right\rbrace , \quad H_2 = H^1\left(\Mc \right),
\]
and let $V = V_1 \times V_2$. The spaces $V_1$, $V_2$ and $V$ equipped by the inner product of $H^1$ are Hilbert spaces. Clearly, $V \hook H$.

Let $A_1: V_1 \to V_1'$, $A_2: V_2 \to V_2'$ be the symmetric linear operators given by the bilinear forms
\begin{gather*}
	(A_1 v, v^\sharp) = a_1(v, v^\sharp) = \int_{\Mc} \mu_v \grad v \cdot \grad v^\sharp + \nu_v \partial_z v \partial_z v^\sharp \, d\Mc,\\
	(A_2 T, T^\sharp) = a_2(T, T^\sharp) = \int_{\Mc} \mu_T \grad T \cdot \grad T^\sharp + \nu_T \partial_z T \partial_z T^\sharp \, d\Mc + \alpha \int_{\Gamma_i} T T^\sharp \, d\Gamma_i .
\end{gather*}
By \cite[Theorem 3.4]{ju2017} $A_1 = P_H \laplace$ in $L(V, V')$. Let
\[
	AU = \begin{pmatrix}
		A_1 v\\
		A_2 T
	\end{pmatrix}.
\]
The operators $A_1, A_2$ and $A$ can be extended to self-adjoint unbounded operators on $H_1, H_2$ and $H$, respectively, see \cite[Lemma 2.4]{lions1992b}. Then we have
\[
	D(A) = \left\lbrace U \in V \mid AU \in H \right\rbrace.
\]

Let $b$ be the trilinear form defined by
\[
	b(U, U^\sharp, U^\flat) = \int_{\Mc} \left[ (v \cdot \grad) v^\sharp + w(v) \partial_z v^\sharp \right] v^\flat + \left[ v \grad T^\sharp + w(v)\partial_z T^\sharp \right] T^\flat \, d\Mc.
\]
From \cite[Lemma 2.1]{petcu2009} we have
\begin{equation}
	\label{eq:b.estimate1}
	\left| b(U, U^\sharp, U^\flat) \right| \leq C \Vert U \Vert \Vert U^\sharp \Vert_{H^2} \Vert U^\flat \Vert, \qquad U, U^\flat \in H^1, U^\sharp \in H^2.
\end{equation}
By \cite[Lemma 3.1]{petcu2009} for $U, U^\sharp \in H^2$ and $U^\flat \in H$ we have
\begin{equation}
	\label{eq:b.estimatel6}
	\left|b\left(U, U^\sharp, U^\flat\right)\right| \leq C \left( \vert v \vert_{L^6} \Vert U^\sharp \Vert^{1/2} \Vert U^\sharp \Vert_{H^2}^{1/2} + \Vert v \Vert^{1/2} \Vert v \Vert_{H^2}^{1/2} \vert \partial_z U^\sharp \vert^{1/2} \Vert \partial_z U^\sharp \Vert^{1/2} \right) \vert U^\flat \vert.
\end{equation}
Using a similar argument as in \cite[Lemma 3.1]{petcu2009} one can establish
\begin{equation}
	\label{eq:b.estimate2}
	\left| b(U, U^\sharp, U^\flat) \right| \leq C \Vert U \Vert^{1/2} \Vert U \Vert_{H^2}^{1/2} \Vert U^\sharp \Vert^{1/2} \Vert U^\sharp \Vert_{H^2}^{1/2} \vert U^\flat \vert, \qquad U, U^\sharp \in H^2, U^\flat \in L^2.
\end{equation}
and an improvement of the estimate \eqref{eq:b.estimate1}
\begin{equation}
\label{eq:b.estimate3}
	\left| b\left(U, U^\sharp, U^\flat\right) \right| \leq C \Vert v \Vert  \Vert U^\sharp \Vert^{1/2} \Vert U^\sharp \Vert_{H^2}^{1/2} \vert U^\flat \vert^{1/2} \Vert U^\flat \Vert^{1/2}, \qquad U, U^\flat \in H^1, U^\sharp \in H^2.
\end{equation}
A similar estimate for $U^\flat \in H^3$ has been established in \cite[Lemma 2.3]{samelson1998}. The form $b$ has the anti-symmetry property
\[
	b(U, U^\sharp, U^\flat) = - b(U, U^\flat, U^\sharp), \qquad U \in V, U^\sharp, U^\flat \in V \cap H^2,
\]
in particular
\begin{equation}
	\label{eq:b.cancellation}
	b(U, U^\sharp, U^\sharp) = 0, \qquad U \in V, U^\sharp \in V \cap H^2.
\end{equation}
We define the bilinear operator $B$ by
\[
	B(U, U^\sharp) = P_H \begin{pmatrix}
		(v \cdot \grad) v^\sharp + w(v) \partial_z v^\sharp\\
		v \grad T^\sharp + w(v) \partial_z T^\sharp
	\end{pmatrix}
\]
and write $B(U) = B(U, U)$.

Let $\Apr: V \to H$ be the linear operator
\[
	\Apr U = P_H \begin{pmatrix}
		-\beta_T g \grad \int_z^0 T z, dz'\\
		0
	\end{pmatrix}.
\]
Clearly $\Apr$ is continuous. We define the linear operator $E: H \to H$ by
\[
	EU = P_H \begin{pmatrix}
		f \vec{k} \times v\\
		0
	\end{pmatrix}
\]
Clearly $E$ is continuous and $(EU, U) = 0$. Let
\begin{equation}
	\label{eq:f.source}
	F_U = P_H \begin{pmatrix}
		F_v\\
		F_T
	\end{pmatrix} \in L^2_{\rm{loc}}(0, \infty; H).
\end{equation}
We denote
\[
	F(U) = \Apr U + EU + F_U.
\] 
To summarize the above, we assume that $F: V \to H$ satisfies
\begin{align}
	\label{eq:F.bnd}
	\int_s^t \vert F(U) \vert^2 \, dr &\leq C \left( \int_s^t \vert F_U \vert^2 + \Vert U \Vert^2 \, dr \right), & &U \in V, 0 \leq s \leq t < \infty,\\
	\label{eq:F.lip}
	\vert F(U) - F(U^\sharp) \vert &\leq C \Vert U - U^\sharp \Vert, & &U \in V,
\end{align}
with the constant $C$ in \eqref{eq:F.bnd} independent of $s, t$.

Let $\Ac_2$, $\Ac_3$ and $\Rc$ be the averaging operators and the remainder defined for $v: \Mc \to \Rb^2$ by
\begin{equation}
	\label{eq:averaging.operators}
	\left(\Ac_2 v\right)(x, y) = \frac{1}{h} \int_{-h}^0 v(x, y, z') \, dz', \quad \left( \Ac_3 v \right)(x, y, z) = \left( \Ac_2 v\right)(x, y), \quad \Rc = I - \Ac_3.
\end{equation}
It is straightforward to check that $\Vert \Ac_3 \Vert_{L(H_1)} \leq 1$ and $\Vert \Ac_2 \Vert_{L(H_1, \Ho)} \leq h^{-1/2}$, $\Rc: H_1 \to H_1$ and $\Vert \Rc \Vert_{L\left(H_1\right)} \leq 2$. Since the spaces $H$ and $\Ho$ have the norm of $L^2\left(\Mc; \Rb^2 \right)$ and $L^2\left( \Mc_0; \Rb^2 \right)$, respectively, we observe that the operators $\Ac_2$, $\Ac_3$ and $\Rc$ remain bounded also if considered with $L^2\left(\Mc\right)$ and $L^2\left(\Mc_0\right)$ in place of $H_1$ and $\Ho$.

Let $\sigma \in \Lip\left(V, L_2\left( \Uc, H \right) \right) \cap \Lip\left(D(A), L_2\left( \Uc, V \right) \right)$, in particular
\begin{align}
	\label{eq:sigma.bnd.H}
	\Vert \sigma(U) \Vert_{L_2(\Uc, H)}^2 &\leq C\left( 1 + \vert U \vert^2 \right) + \eta_0 \Vert U \Vert^2, & U &\in H,\\
	\label{eq:sigma.bnd.V}
	\Vert \sigma(U) \Vert_{L_2(\Uc, V)}^2 &\leq C \left( 1 + \Vert U \Vert^2 \right) + \eta_1 \vert A U \vert^2, & U &\in V,\\
	\label{eq:sigma.lip.H}
	\Vert \sigma(U) - \sigma(U^\sharp) \Vert_{L_2(\Uc, H)}^2 &	\leq C \Vert U - U^\sharp \Vert^2, & U, U^\sharp &\in V,\\
	\label{eq:sigma.lip.V}
	\Vert \sigma(U) - \sigma(U^\sharp) \Vert_{L^2(\Uc, V)}^2 &\leq C \Vert U - U^\sharp \Vert^2 + \gamma \vert AU - AU^\sharp \vert^2, & U, U^\sharp &\in D(A),
\end{align}
for some $\eta_0, \eta_1, \gamma \geq 0$. Due to the nature of the estimates in Sections \ref{sect:ldp} and \ref{sect:mdp} we also assume that $\sigma$ satisfies the following structural assumptions
\begin{align}
	\label{eq:sigma1.l6}
	\sum_{k=1}^\infty \left| \Rc \sigma_1(U) e_k \right|_{L^6}^2 &\leq C \left( 1 + \vert \Rc v \vert_{L^6}^2 \right),\\
	\label{eq:sigma.rc.t}
	\sum_{k=1}^\infty \left| \sigma_2(U) e_k \right|_{L^6}^2 &\leq C \left( 1 + \vert T \vert_{L^6}^2 \right),\\
	\label{eq:sigma.vb}
	\Vert \Ac_2 \sigma_1(U) \Vert_{L_2(\Uc, \Vo)}^2 &\leq C \left( 1 + \Vert U \Vert^2 \right) + \eta_2 \vert A_S \Ac_2 v \vert_{\Ho}^2,\\
	\label{eq:sigma.vz}
	\Vert \partial_z \sigma_i(U) \Vert_{L_2(\Uc, L^2)}^2 &\leq C \left( 1 + \Vert U \Vert^2 \right) + \eta_3 \vert \grad_3 \partial_z U_i \vert^2.
\end{align}
with $\eta_2, \eta_3 \geq 0$. A non-trivial example of the noise term $\sigma$ depending on the horizontal gradient of the vertically averaged velocity satisfying the assumptions (\ref{eq:sigma.bnd.H}-\ref{eq:sigma.vz}) can be found in \cite[Section 2.5]{brzezniakslavik}.

Using the notation from the above we are able to rewrite the equations (\ref{eq:pe.v.reform}-\ref{eq:pe.T.reform}) in the abstract form
\begin{equation}
	\label{eq:pe.abstract}
	dU + \left[ AU + B(U) + F(U) \right] \, dt = \sigma(U) \, dW, \qquad U(0) = u_0.
\end{equation}

In \cite{guillengonzalez2001} the anisotropic spaces have been used do establish estimates on the non-linear term $B$. For $1 \leq p, q < \infty$ we denote
\[
	\vert v \vert_{L^q_x L^p_z} = \left( \int_{\Mc_0} \left( \int_{-h}^0 \left( \vert v_1 \vert^p + \vert v_2 \vert^p \right) \, dz \right)^{q/p} \, d\Mc_0 \right)^{1/q}.
\]

\begin{lemma}
\label{lemma:anisotropic.estimates}
The following anisotropic estimates hold with constant $C$ depending only on $\Mc$:
\begin{align}
	\label{eq:anis.1}
	\vert v \vert_{L^q_x L^2_z} &\leq C \vert v \vert^{2/q} \Vert v \Vert^{1-2/q}, & &v \in H^1, q \geq 2,\\
	\label{eq:anis.2}
	\vert v \vert_{L^q_x L^2_z} &\leq C \vert v \vert_{L^6}^{6/q} \Vert v^3 \Vert^{1/3 - 2/q}, & &v^3 \in H^1, q \geq 6,\\
	\label{eq:anis.3}
	\esssup_{z \in (-h, 0)} \vert v(\cdot, z) \vert_{L^2(\Mc_0)} \equiv \vert v \vert_{L^\infty_z L^2_x} &\leq C \vert v \vert^{1/2} \Vert v \Vert^{1/2}, & &v \in H^1,\\
	\label{eq:anis.4}
	\vert v^5 \vert_{L^3_x L^2_z} &\leq C \vert v \vert_{L^6} \Vert v^3 \Vert^{4/3}, & &v^3 \in H^1,\\
	\label{eq:anis.5}
	\vert v^2 \vert_{L^4_x L^3_z} &\leq C \vert v \vert_{L^6}^{3/2} \Vert v^3 \Vert^{1/6}, & &v^3 \in H^1.
\end{align}
\end{lemma}

\begin{proof}
The estimate \eqref{eq:anis.1} can be shown by repeating the argument from the proof of \cite[Lemma 3.1]{petcu2009}. The inequality \eqref{eq:anis.3} has been established in \cite[Lemma 3.3(a)]{guillengonzalez2001}. The estimate \eqref{eq:anis.2} follows from \eqref{eq:anis.1}. For $s = 1-6/q$ we use the H\"{o}lder inequality and \eqref{eq:anis.1} and obtain
\begin{equation*}
	\vert v \vert_{L^q_x L^2_z} \leq \vert v^3 \vert_{L^{q/3}_x L^2_z}^{1/3} \leq C \vert v^3 \vert^{(1-s)/3} \Vert v^3 \Vert^{s/3} = C \vert v \vert_{L^6}^{1-s} \Vert v^3 \Vert^{s/3}.
\end{equation*}
Regarding \eqref{eq:anis.4} we use the Minkowski inequality, the Gagliardo-Nirenberg inequality in 2D and the estimate \eqref{eq:anis.3} to deduce
\begin{align*}
	\vert v^5 \vert_{L^3_x L^2_z}^2 &= \left( \int_{\Mc_0} \left(\int_{-h}^0 \vert v \vert^{10} \, dz \right)^{3/2} \, d\Mc_0 \right)^{2/3} \leq \int_{-h}^0 \left( \int_{\Mc_0} \vert v\vert^{15} \, d\Mc_0 \right)^{2/3} = \int_{-h}^0 \vert v^3 \vert_{L^5_x}^{10/3} \, dz\\
	&\leq C \int_{-h}^0 \vert v^3 \vert_{L^2_x}^{4/3} \vert \grad (v^3) \vert_{L^2_x}^{2} + \vert v^3 \vert_{L^2_x}^{10/3} \, dz \leq C \vert v^3 \vert_{L^\infty_z L^2_x}^{4/3} \Vert v^3 \Vert^2\\
	&\leq C \vert v^3 \vert^{2/3} \Vert v^3 \Vert^{8/3} \leq C \vert v \vert_{L^6}^2 \Vert v^3 \Vert^{8/3}.
\end{align*}
The final inquality \eqref{eq:anis.5} can be established similarly as \eqref{eq:anis.4}.
\end{proof}

For the sake of completeness let us recall the Burkholder-Davis-Gundy inequality. Let $X$ be a separable Hilbert space and $r \geq 2$. Let $W$ be a cylindrical Wiener process with RKHS $\Uc$ defined on a probability space $(\Omega, \Fc, \Pb)$. Then for all $\Phi \in L^2\left(\Omega, L^2\left(0, T; L_2\left(\Uc, X\right)\right)\right)$
\begin{equation}
	\label{eq:bdg}
	\Eb \sup_{t \in \left[0, T\right]} \left| \int_0^t \Phi \, dW \right|^r_X \leq C_{BDG} \, \Eb \left( \int_0^T \Vert \Phi \Vert_{L_2(\Uc, X)}^2 \, dt \right)^{r/2},
\end{equation}
where the constant $C_{BDG}$ depends only on $r$. For proof see e.g.\ \cite[Theorem 3.28, p.\ 166]{karatzas1991}.

\subsection{Definition of a solution}

Let $(\Omega, \Fc, \Fb, \Pb)$, $\Fb = (\Fc_t)_{t \geq 0}$, be a stochastic basis satisfying the usual conditions. We will consider only solutions strong both in the stochastic and PDE sense.

\begin{definition}
\label{def:solution}
A progressively measurable stochastic process $U: \Omega \times [0, \infty) \to V$ is a global solution of the equation \eqref{eq:pe.abstract} if there exists a nondecreasing sequence of $\Fb$-stopping times $\tau_N$ such that
\begin{enumerate}
	\item for all $N \in \Nb$ and $t \geq 0$
	\begin{equation}
	\label{eq:solution.regularity}
		U\left( \cdot \wedge \tau_N \right) \in L^2\left(\Omega; C\left([0, t], V\right)\right), \qquad \mathds{1}_{[0, \tau_N]}(\cdot) U \in L^2\left(\Omega; L^2\left(0, t; D(A) \right)\right),
	\end{equation}
	\item for all $N \in \Nb$ and $t \geq 0$ the the stopped process satisfies the following equation $\Pb$-a.s.\ in $H$
	\begin{equation}
		\label{eq:solution.def}
		U\left(t \wedge \tau_N\right) + \int_0^{t \wedge \tau_N} AU + B(U) + \Apr U + EU - F_U \, ds = U(0) + \int_0^{t \wedge \tau_N} \sigma(U) \, dW.
	\end{equation}
	\item $\tau_N \to \infty$ $\Pb$-almost surely.
\end{enumerate}
\end{definition}

The following well-posedness result has been established in \cite[Theorem 2.6]{brzezniakslavik}.

\begin{theorem}
\label{thm:global.existence}
Let $u_0 \in V$. Let $F_U$ satisfy \eqref{eq:f.source} and $F_T \in L^2\left( \Omega; L^2\left(0, t; L^2(\Gamma_i) \right) \right)$. Let $\sigma$ satisfy (\ref{eq:sigma.bnd.H}-\ref{eq:sigma.vz}). If the constants $\gamma$, $\eta_i$, $i = 1, 2, 3, 4$, in \eqref{eq:sigma.bnd.H}, \eqref{eq:sigma.bnd.V}, \eqref{eq:sigma.lip.V}, \eqref{eq:sigma.vb} and \eqref{eq:sigma.vz} are sufficiently small, then there exists a unique global solution of the equation \eqref{eq:pe.abstract}. Moreover, if $\tau_N$ is the sequence of stopping times from Definition \ref{def:solution}, the solution $U$ satisfies
\[
	\mathds{1}_{[0, \tau_N]} U \in L^p\left( \Omega; L^\infty(0, t; V) \right), \qquad \mathds{1}_{[0, \tau_N]} \Vert U \Vert^{p-2} \vert AU \vert^2 \in L^1\left( \Omega; L^1(0, t) \right),
\]
for all $N \in \Nb$, $p \geq 2$ and $t > 0$.
\end{theorem}

In the subsequent sections we study equations with stochastic forcing $\sqrt{\varepsilon} \sigma(U) \, dW$ for $\varepsilon \in (0, \varepsilon_0]$ and study the limiting behaviour for $\varepsilon \to 0+$. Therefore, assuming that $\varepsilon_0$ is sufficiently small, the requirement on smallness of the constants $\gamma$, $\eta_i$, $i = 1, 2, 3, 4$, in the theorem above may be relaxed.

\subsection{Large deviations principle}

Let $\left( \Omega, \Fc, \Pb \right)$ be a complete probability space and let $X$ be a separable Banach space and let $W$ be a cylindrical Wiener process with reproducing kernel Hilbert space $\Uc$.

\begin{definition}
A function $I: X \to [0, \infty]$ is called a \emph{rate function} if it is lower semicontinuous. If for all $M > 0$ the set $\lbrace U \in X \mid I(U) \leq M \rbrace$ is compact a rate function $I$ is called a \emph{good rate function}.
\end{definition}

\begin{definition} 
Let $\varepsilon_0 > 0$ and let $\lbrace \Ue \rbrace_{\varepsilon \in (0, \varepsilon_0]}$ be $X$-valued stochastic processes, $U^0 \in X$ and let $I$ be a rate function. We say that the processes $\lbrace \Ue \rbrace_{\varepsilon \in (0, \varepsilon_0]}$ satisfy
\begin{enumerate}
	\item the \emph{large deviations principle} (LDP) with the rate function $I$ if for each $A \in \Borel(X)$
	\[
		- \inf_{x \in A^o} I(x) \leq \liminf_{\varepsilon \to 0+} \varepsilon \log \Pb(\lbrace U^\varepsilon \in A \rbrace)\\
		\leq \limsup_{\varepsilon \to 0+} \varepsilon \log \Pb(\lbrace U^\varepsilon \in A \rbrace) \leq - \inf_{x \in \bar{A}} I(x),
	\]
	\item the \emph{moderate deviations principle} (MDP) the with rate function $I$ if the processes
	\[
		R^\varepsilon = \lbrace (U^\varepsilon - U^0)/(\sqrt{\varepsilon} \lambda(\varepsilon) \rbrace_{\varepsilon \in (0, \varepsilon_0]}
	\]
	satisfy the LDP with the rate function $I$.
\end{enumerate}
\end{definition}

It is well-known that if $\Uc_0$ is a separable Hilbert space such that the embedding $\Uc \hook \Uc_0$ is Hilbert-Schmidt, the cylindrical Wiener process $W$ has continuous trajectories in $\Uc_0$. Let $\Gc^\varepsilon: C([0, t], \Uc_0) \to X$ be a measurable map such that $\Ue = \Gc^\varepsilon(W(\cdot))$. The following theorem from \cite[Theorem 4.4]{budhiraja2000} establishes a sufficient condition for the LDP.

\begin{theorem}
\label{thm:ldp.abstract}
Let $\Gc^0: C\left( [0, T], \Uc_0 \right) \to X$ be a measurable map satisfying the following conditions:
\begin{enumerate}
	\item Let $M > 0$ and let $\lbrace h_\varepsilon \mid \varepsilon \in (0, \varepsilon_0] \rbrace \subseteq \Ac_M$ be such that $h_\varepsilon \to h$ in distribution for some $\Tc_M$-valued random variable $h$. Then
	\begin{equation}
		\label{eq:ldp.cond.1}
		\Gc^\varepsilon\left( W(\cdot) + \frac{1}{\sqrt{\varepsilon}} \int_0^\cdot h_\varepsilon(s) \, ds \right) \to \Gc^0\left( \int_0^\cdot h(s) \, ds \right)	 \quad \text{in distribution}.
	\end{equation}
	\item The set
	\begin{equation}
	\label{eq:ldp.cond.2}
		K_M = \left\{ \Gc^0 \left( \int_0^\cdot h(s) \, ds \right) \mid h \in \Tc_M \right\} 	
	\end{equation}
	is compact in $X$ for all $M  > 0$.
\end{enumerate}
Then $Y^\varepsilon$ satisfies the large deviations principle with a good rate function
\begin{equation}
	\label{eq:ldp.rate}
	I(U) = \inf \left\lbrace \frac12 \int_0^t \vert h \vert_{\Uc}^2 \, ds \mid h \in L^2(0, t; \Uc) \ \text{s.t.} \ U = \Gc^0\left(\int_0^\cdot h \, ds\right) \right\rbrace.
\end{equation}
\end{theorem}

\section{Large deviations principle}
\label{sect:ldp}

The proofs in this section are modifications of the ones in \cite{dong2017}, the modifications consisting in considering less regular noise term $\sigma$ and taking into account the definition of solution of the stochastic primitive equations. We will concentrate mainly on the differences between our setting and the one in \cite{dong2017}.

Let $U^\varepsilon$ be the solution of the equation
\begin{equation}
	\label{eq:ldp.main}
	d\Ue + \left[ A\Ue + B\left( \Ue \right) + F\left( \Ue \right) \right] \, dt = \sqrt{\varepsilon} \sigma\left( \Ue \right) \, dW, \qquad \Ue(0) = u_0.
\end{equation}
The equation \eqref{eq:ldp.main} is well-posed by Theorem \ref{thm:global.existence} for sufficiently small $\varepsilon > 0$. By the result of \cite{rockner2008} there exists a measurable map $\Gc^\varepsilon: C([0, t], \Uc_0) \to C([0, t], V) \cap L^2(0, t; D(A))$ such that $U^\varepsilon = \Gc^\varepsilon(W(\cdot))$.

\subsection{Skeleton equation}
\label{sect:ldp.skeleton}

Let $h \in L^2\left( 0, T; \Uc \right)$. By the \emph{skeleton equation} we understand the equation
\begin{equation}
	\label{eq:skeleton}
	dU_h + \left[ AU_h + B(U_h) + F(U_h) \right] \, dt = \sigma\left( U_h \right) h \, dt, \qquad U_h(0) = u_0.
\end{equation}

\begin{proposition}
\label{prop:skeleton.existence}
The equation \eqref{eq:skeleton} is well-posed. Moreover, the solution $U_h$ satisfies
\begin{equation}
	\label{eq:skeleton.bounds}
	\sup_{s \in [0, t]} \Vert U_h \Vert^2 + \int_0^t \vert AU_h \vert^2 \, ds \leq C_{t, M, U_0}.
\end{equation}
\end{proposition}

\begin{proof}
The proof runs along the same lines as in \cite[Section 5]{dong2017}, that is by using the approaches by \cite{petcu2009} and \cite{cao2007} for the local and global existence results, respectively, with only minor additional complications in regards to the presence of the control term $\sigma(U_h) h$. Similar estimates will be seen in Section \ref{sect:mdp.preliminary} for a more complicated stochastic equation.
\end{proof}

Before we proceed to the compactness argument, let us return to the regularity assumption $U_h \in L^\infty(0, t; H^2)$ from Theorem \ref{thm:ldp}. Existence of local regular solutions of \eqref{eq:skeleton} has been established in \cite{caolititi2014} for $\sigma \equiv 0$ by a fixed point argument. In the following Proposition we introduce additional assumptions on $\sigma$ to ensure $U_h \in L^\infty(H^2)$. These assumptions are more or less a technical issue in the sense that the LDP result from Theorem \ref{thm:ldp} depends only on the regularity of $U_h$ and not on the particular additional assumptions formulated in the following Proposition.

\begin{proposition}
Let $\sigma: [0, \infty) \times V \to H$ be such that the bounds (\ref{eq:sigma.bnd.H}-\ref{eq:sigma.bnd.V}) and (\ref{eq:sigma1.l6}-\ref{eq:sigma.vz}) are satisfied with $\sigma(t, U)$ instead of $\sigma(U)$ with constants $C, \eta_i$, $i=0, 1, 2, 3$, independent of $t$. Let $\gamma_H, \gamma_V: [0, \infty) \to [0, \infty)$ be continuous functions such that $\gamma_H(0) = \gamma_V(0) = 0$ and
\begin{align}
	\label{eq:sigma.lip.tH}
	\Vert \sigma(t, U_1) - \sigma(t, U_2) \Vert_{L_2\left( \Uc, H\right)} &\leq \gamma_H(t) \Vert U_1 - U_2 \Vert, & U_1, U_2 &\in V,\\
	\label{eq:sigma.lip.tV}
	\Vert \sigma(t, U_1) - \sigma(t, U_2) \Vert_{L_2\left( \Uc, V\right)} &\leq \gamma_V(t) \Vert U_1 - U_2 \Vert_{H^2}, & U_1, U_2 &\in H^2.
\end{align}
Then the solution $U_h$ of \eqref{eq:skeleton} satisfies
\[
	U_H \in L^\infty\left(0, t; H^2\right) \cap L^2\left(0, t; H^3\right)
\]
for all $t > 0$, $h \in L^2(0, t; \Uc)$ and $u_0 \in V \cap H^2$.
\end{proposition}

\begin{proof}
We briefly comment on the extension of the fixed point argument from \cite[Section 2]{caolititi2014}. For simplicity we will not include all the details such as the reflection argument, which will lead to ignoring some of the spatial symmetries in the system. Let for $t > 0$
\begin{gather*}
	X_t = \left\lbrace U = (v, T) \in C\left([0, t], H^2 \right) \cap L^3\left(0, t; H^3\right) \mid \div \Ac_3 v = 0 \right\rbrace,\\
	\Vert U \Vert_{X_t}^2 = \sup_{s \in [0, t]} \Vert U(s) \Vert^2_{H^2} + \int_0^t \Vert U(s) \Vert_{H^3}^2 \, ds.
\end{gather*}
Let
\begin{align*}
	D(U) = -B(U) - \Apr U - EU - F_U, \qquad U \in X_t.
\end{align*}
Using the decomposition into barotropic and baroclinic modes as in \cite{cao2007} and standard regularity theory for the Stokes equation and parabolic equations one can show that the equation
\begin{equation}
	\label{eq:v.regularity}
	d\Vc + A\Vc \, dt = \left( D(U) + \sigma(U)h \right) \, dt, \qquad \Vc(0) = U(0),
\end{equation}
has a solution in $X_t$ for all $U \in X_t$. Let $h \in L^2(0, t; \Uc)$ with $\Vert h \Vert_{L^2(0, t; \Uc)} \leq M$ for some $M \geq 0$. Let $U_1, U_2 \in X_t$ such that $\Vert U_1 \Vert_{X_t}, \Vert U_2 \Vert_{X_t} \leq K$ for some $K \geq 0$ and let $\Vc_1, \Vc_2 \in X_t$ be the respective solutions of \eqref{eq:v.regularity} with $U = U_1$ and $U_2$, respectively. Subtracting the equations for $\Vc_1$ and $\Vc_2$, testing by $\Vc = \Vc_1 - \Vc_2$ and $- \grad \laplace \Vc$ and using the bounds on $D(U)$ from \cite[Proposition 2.2]{caolititi2014} and (\ref{eq:sigma.lip.tH}-\ref{eq:sigma.lip.tV}) we obtain the estimate
\[
	\sup_{s \in [0, t]} \Vert \Vc \Vert^2_{H^2} + \int_0^t \Vert \grad \Vc \Vert_{H^2}^2 \, ds \leq C \Vert U_1 - U_2 \Vert_{X_t}^2 \left( K^2 t^{1/2} + M^2 \sup_{s \in [0, t]} \max \lbrace \gamma_H(s), \gamma_V(s) \rbrace \right).
\]
It is now easy to see that thanks to continuity of $\gamma_H$ and $\gamma_V$ and $\gamma_H(0) = \gamma_V(0) = 0$ one can get a contraction. The rest of the fixed point argument follows easily as in \cite[Propositions 2.1 and 2.3]{caolititi2014}.

The global existence follows from apriori estimates similar to those in \cite[Section 3]{ju2017.1}.
\end{proof}

\begin{proposition}
\label{prop:skeleton.convergence}
Let $\lbrace h_n \rbrace_{n=1}^\infty \subseteq \Tc_M$ and let $U^n = U_{h_n}$ be the respective solutions of \eqref{eq:skeleton}. Then there exist $h \in L^2(0, t; \Uc)$ and a subsequence (for simplicity not relabelled) such that $h_n \harp h$ in $L^2(0, t; \Uc)$ and $U^n \to U_h$ in $L^\infty(0, t; V) \cap L^2(0, t; D(A))$.

Moreover, there exists a measurable map $\Gc^0: C([0, t], \Uc_0) \to C([0, t], V) \cap L^2(0, t; D(A))$ such that $U_h = \Gc^0(\int_0^\cdot h_s \, ds)$.
\end{proposition}

\begin{proof}
Recalling the growth estimate of $\sigma$ \eqref{eq:sigma.bnd.H} in $L_2(\Uc; H)$ we may repeat the argument from \cite[Proposition 5.3]{dong2017} to show that there exists a (not relabelled) subsequence $U^n$ and $\Ut \in L^\infty(0, t; V) \cap L^2(0, t; D(A))$ such that 
\begin{equation}
	\label{eq:skeleton.weak.convergence}
	U^n \to \Ut \ \text{in} \ L^2(0, t; V), \quad U^n \harpstar \Ut \ \text{in} \ L^\infty(0, t; V), \quad U^n \harpstar \Ut \ \text{in} \ L^2(0, t; D(A)).
\end{equation}
The continuity of $\Ut$ can be show using a maximal regularity type argument based on the Lions-Magenes lemma, see e.g.\ \cite[Lemma 1.2, Chapter 3]{temam1979}, similarly as in \cite[Section 4]{ju2007} for the deterministic case or \cite[Section 7.3]{debussche2011} for the stochastic case. The stochastic case with the same assumptions on $\sigma$ can be found in \cite[Proposition 3.3]{brzezniakslavik}.

Next we show $\Ut = U_h$. We proceed similarly as in e.g.\ \cite[Section 7]{debussche2011}. Let $U^\sharp \in D(A)$ be fixed. Since $A$ is self-adjoint, we may use the Cauchy-Schwartz inequality and the first convergence in \eqref{eq:skeleton.weak.convergence} to get
\[
	\left| \int_0^t \left( AU^n - A\Ut, U^\sharp \right) \, ds \right| \leq C_t \Vert U^\sharp \Vert \left( \int_0^t \Vert U^n - \Ut \Vert^2 \, ds \right)^{1/2} \to 0.
\]
By the Lipschitz continuity of F in \eqref{eq:F.lip} and the first convergence in \eqref{eq:skeleton.weak.convergence} we have
\[
	\left| \int_0^t \left( F(U^n) - F(\Ut), U^\sharp \right) \, ds \right| \leq C_t \vert U^\sharp \vert \left( \int_0^t \Vert U^n - \Ut \Vert^2 \, ds \right) \to 0.
\]
By the estimate \eqref{eq:b.estimate1} and the first two convergences \eqref{eq:skeleton.weak.convergence} we deduce
\begin{align*}
	\left| \int_0^t \left( B(U^n) - B(\Ut), U^\sharp \right) \, ds \right| &\leq C \int_0^t \Vert U^n \Vert  \Vert U^n - \Ut \Vert \vert A U^\sharp \vert + \Vert U^n - \Ut \Vert \Vert \Ut \Vert \vert AU^\sharp \vert \, ds\\
	&\leq C_t \left( \sup_{s \in [0, t]} \Vert U^n \Vert + \sup_{s \in [0, t]} \Vert \Ut \Vert \right) \left( \int_0^t \Vert U^n - \Ut \Vert^2 \, ds \right)^{1/2} \to 0.
\end{align*}
By the triangle inequality we have
\begin{multline*}
	\left| \int_0^t \left( \sigma(U^n) h_n - \sigma(\Ut) h, U^\sharp \right) \, ds \right|\\
	\leq \int_0^t \left| \left( \left[ \sigma(U^n) - \sigma(\Ut) \right] h_n, U^\sharp \right) \right| \, ds + \left| \int_0^t \left( \sigma(\Ut)\left[ h^n - h \right], U^\sharp \right) \, ds \right| =  \left| I_1^n \right|  + \left| I_2^n \right|.
\end{multline*}
The first integral can be estimated directly using the Lipschitz continuity \eqref{eq:sigma.lip.H} of $\sigma$ in $L_2(\Uc, H)$, the boundedness of $h_n$ and the first convergence in \eqref{eq:skeleton.weak.convergence} by
\[
	\left| I_1^n \right| \leq C \vert U^\sharp \vert \left( \int_0^t \vert h_n \vert^2_{\Uc} \, ds \right)^{1/2} \left( \int_0^t \Vert U^n - \Ut \Vert^2 \, ds \right)^{1/2} \to 0.
\]
Since $\sigma(\Ut): L^2\left( 0, t; \Uc \right) \to L^2\left(0, t; H \right)$ is a linear bounded operator, the integral $I_2^n$ converges to zero using the weak convergence $\sigma(\Ut) [h_n - h] \harp 0$ in $L^2\left( 0, t; H \right)$. Collecting the above we get
\begin{multline*}
	\left( U^n(t), U^\sharp \right) + \int_0^t \left( AU^n, U^\sharp \right) + \left( B(U^n), U^\sharp \right) + \left( F(U^n), U^\sharp \right) - \left( \sigma(U^n), U^\sharp \right)\\
	\to \left( \Ut(t), U^\sharp \right) + \int_0^t \left( A\Ut, U^\sharp \right) + \left( B(\Ut), U^\sharp \right) + \left( F(\Ut), U^\sharp \right) - \left( \sigma(\Ut), U^\sharp \right)
\end{multline*}
and thus $\Ut$ satisfies the equation \eqref{eq:skeleton} in $D(A)'$. Using the density of $D(A)$ in $H$ we get that $\Ut$ satisfies \eqref{eq:skeleton} in $H$ and thus by uniqueness $\Ut = U_h$.

It remains to establish the strong convergence in $L^\infty(0, t; V) \cap L^2(0, t; D(A))$. To that end we adapt the proof of \cite[Theorem 5.8]{dong2017}. Let $w^n = U^n - U_h$. Clearly, $w_n$ satisfies
\[
	dw^n + \left[ Aw^n + B(U^n) - B(U_h) + F(U^n) - F(U_h) \right] \, dt = \left[ \sigma(U^n) h_n - \sigma(U_h) \right] \, dt, \quad w^n(0) = 0.
\]
By the Ito lemma, see e.g.\ \cite[Theorem A.1]{brzezniakslavik}, we have
\begin{multline*}
	d\Vert w^n \Vert^2 + 2\vert Aw^n \vert^2 \, dt = - 2 \left( A^{1/2}\left[ B(U^n) - B(U_h) \right], A^{1/2} w^n \right) \, dt\\
	- 2 \left[ \left( A^{1/2}\left[ F(U^n) - F(U_h) \right], A^{1/2} w^n \right) + \left( A^{1/2}\left[ \sigma(U^n)h_n - \sigma(U_h)h \right], A^{1/2} w^n \right) \right] \, dt.
\end{multline*}
Let $\eta > 0$ be fixed and precisely determined later and let $s \in [0, t]$. Since $A$ is self-adjoint, we recall the Lipschitz continuity of $F$ in \eqref{eq:F.lip} and use the Cauchy-Schwartz and the Young inequalities to obtain
\[
	\left| 2 \int_0^s \left( A^{1/2}\left[ F(U^n) - F(U_h) \right], A^{1/2} w^n \right) \, dr \right| \leq \frac{\eta}{3} \int_0^s \vert A w^n \vert^2 \, ds + C_\eta \int_0^t \Vert w^n \Vert^2 \, dr.
\]
Next we use the estimate \eqref{eq:b.estimate2} and the boundedness of $U^n$ in $C([0, t], V)$ to deduce
\begin{align*}
	\bigg| 2 \int_0^s &\left( A^{1/2}\left[ B(U^n) - B(U_h) \right], A^{1/2} w^n \right) \, dr \bigg|\\
	&\leq \int_0^s \left| \left( B(U^n, U^n - U_h), Aw^n \right) \right| \, dr + \int_0^s \left| \left( B(U^n - U_h, U_h), Aw^n \right) \right| \, dr\\
	&\leq \frac{\eta}{3} \int_0^s \vert Aw^n \vert^2 \, dr + C_\eta \int_0^s \vert AU^n \vert^2 \Vert w^n \Vert^2 \, dr.
\end{align*}
The final term is treated using the idea from \cite[Theorem 5.8]{dong2017} using the additional regularity of the solution of the skeleton equation \eqref{eq:skeleton}. By the Lipschitz continuity of $\sigma$ \eqref{eq:sigma.lip.H} in $L_2(\Uc, H)$ and the growth estimate of $\sigma$ \eqref{eq:sigma.bnd.V} in $L_2(\Uc, V)$ we obtain
\begin{align*}
	\bigg| \int_0^s &\left( A^{1/2}\left[ \sigma(U^n)h_n - \sigma(U_h)h \right], A^{1/2} w^n \right) \, dr\bigg|\\
	&\leq \left| \int_0^s \left( A^{1/2}\left[ \sigma(U^n) - \sigma(U_h)\right] h_n, A^{1/2} w^n \right) \, dr \right| + \left| \int_0^s \left( A^{1/2} \sigma(U_h)\left[ h_n - h \right], A^{1/2} w^n \right) \, dr\right| \\
	&\leq \frac{\eta}{3} \int_0^s \vert Aw^n \vert^2 \, dr + C_\eta \int_0^s \vert h_n \vert_{\Uc}^2 \Vert w^n \Vert^2 \, dr\\
	&\hphantom{\leq \ } + C\sup_{r \in [0, s]}\left( 1 + \Vert U_h \Vert_{H^2}\right) \left( \int_0^s \vert h_n - h \vert_{\Uc}^2 \, dr \right)^{1/2} \left( \int_0^s \Vert w^n \Vert^2 \, dr \right)^{1/2}.
\end{align*}
Collecting the estimates above and choosing $\eta > 0$ sufficiently small we get
\begin{multline}
	\label{eq:sk.compactness.estimate}
	\Vert w^n(s) \Vert^2 + \int_0^s \vert Aw^n \vert^2 \, ds \leq C_h \left( \int_0^s \Vert w^n \Vert^2 \, dr \right)^{1/2}\\
	+ C_\eta \int_0^s \Vert w^n \Vert^2 \left( 1 + \vert AU^n \vert^2 + \vert AU_h \vert^2 + \vert h_n \vert^2_{\Uc} \right) \, ds.
\end{multline}
for all $s \in [0, t]$. By the Gronwall lemma and the bounds in \eqref{eq:skeleton.weak.convergence} we get
\begin{equation}
	\label{eq:sk.compactness.gronwall}
	\Vert w^n(s) \Vert^2 \leq C_{h, \lbrace h_k \rbrace, u_0} \left( \int_0^s \Vert w^n \Vert^2 \, dr \right)^{1/2},
\end{equation}
where the constant $C_h$ depends only on $h$, the bound of $\lbrace h_k \rbrace_{k=1}^\infty$ in $L^2(0, t; \Uc)$ and $u_0$, and is independent of $n$. Recalling that $w^n \to 0$ in $L^2(0, t; V)$ by \eqref{eq:skeleton.weak.convergence}, we may pass to the limit in \eqref{eq:sk.compactness.gronwall} and obtain strong convergence in $C([0, t]; V)$. Strong convergence in $L^2(0, t; H^2)$ then follows immediately from \eqref{eq:sk.compactness.estimate}.

The measurable map $\Gc^0: C([0, t], \Uc_0) \to C([0, t], V) \cap L^2(0, t; D(A))$ is then defined by
\begin{equation}
	\label{eq:gc.0}
	\Gc^0(h) = \begin{cases}
		U_{\tilde{h}}, & \text{if} \ \tilde{h} = \int_0^\cdot h(s) \, ds \ \text{for some} \ h \in L^2(0, t; \Uc),\\
		0, & \text{otherwise},
	\end{cases}
\end{equation}
as in \cite{dong2017}.
\end{proof}

\subsection{Preliminary estimate}
\label{sect:ldp.preliminary}

Let $h_\varepsilon \in L^2(0, t; \Uc)$. The equation
\begin{equation}
	\label{eq:ldp.equation}
	d\Ue + \left[ A\Ue + B(\Ue, \Ue) + F(\Ue) \right] \, dt = \sigma(\Ue) h_\varepsilon \, dt + \sqrt{\varepsilon} \sigma(\Ue) \, dW, \quad \Ue(0) = u_0.
\end{equation}
will play an important role in the proof of Theorem \ref{thm:ldp}. In the paper \cite{dong2017} the existence and uniqueness of \eqref{eq:ldp.equation} is established by the means of the Girsanov theorem. However, for the argument in the proof of \cite[Theorem 1.1, Step 2]{dong2017} to go through as described one needs additional information of the behaviour of $A\Uhe$.

Existence of global pathwise solutions $\Ue$ essentially follows from \cite{brzezniakslavik}, which is inspired by the arguments of \cite{cao2007,debussche2011,debussche2012}. In order to avoid repetition we will only state the necessary preliminary result. The proof can be established following a series of estimates similar to the ones in \cite[Section 4]{brzezniakslavik}, see also Section \ref{sect:mdp.preliminary} for estimates of the same kind for a more complicated equation.

\begin{proposition}
\label{prop:ldp.V.estimate}
Let $U_{\varepsilon}$ be the solution of \eqref{eq:ldp.equation} and let $\tau_K^{U, \varepsilon}$ be the stopping time defined for $K > 0$ and $\varepsilon \in (0, \varepsilon_0]$ by
\[
	\tau_K^{U, \varepsilon, p} = \inf\left\lbrace s \geq 0 \mid \int_0^s \Vert U_\varepsilon \Vert^{p-2} \vert AU_\varepsilon \vert^2 \, dr \geq K \right\rbrace.
\]
Then $\tau_K^{U, \varepsilon} \to \infty$ $\Pb$-almost surely. Moreover, for $\tilde{t} > 0$ one has
\[
	\lim_{K \to \infty} \Pb \left( \left\lbrace \tau_K^{U, \varepsilon} \leq \tilde{t} \right\rbrace \right) = 0 \ \text{uniformly w.r.t.\ } \varepsilon \in (0, \varepsilon_0].
\]
\end{proposition}

\subsection{Proof of Theorem \ref{thm:ldp}}
\label{sect:ldp.proof}

We follow the proof presented in \cite[Theorem 1.1]{dong2017}. There are two minor technical issues in the original proof that deserve clarification. First, in \cite{dong2017} it is assumed that the solution of the stochastic 3D primitive equations has the regularity
\[
	U \in L^2 \left( \Omega; C\left( [0, t], V \right) \right) \cap L^2\left( \Omega; L^2\left( 0, t; D(A) \right) \right) \ \text{for all} \ t \geq 0,
\]
but the stronger integrability in $\Omega$ compared to Definition \ref{def:solution} has not been established. This regularity allows the authors to prove the convergence
\[
	Z^\varepsilon \to 0 \quad \text{in} \quad L^2\left(\Omega; C\left( [0, t], V \right) \cap L^2\left( 0, t; D(A) \right) \right),
\]
see \eqref{eq:ldp.z.definition} for the definition of $Z^\varepsilon$. This issue can be resolved by establishing the discussed convergence in probability, see \eqref{eq:ldp.z.convergence}. Second, after using the Skorokhod theorem the existence of the solution of the equation for $X_{\tilde{h}_\varepsilon}$, in the notation used here denoted by $R^\varepsilon$, see \eqref{eq:ldp.r.definition}, is not discussed. As we will see in Section \ref{sect:mdp.preliminary} if we were to establish the existence using the decomposition into the barotropic and baroclinic mode and similar estimates in \cite{cao2007}, we would need additional assumptions on regularity of the solution of $Z^\varepsilon$ which we do not have. This is resolved by using the Bensoussan trick from \cite{bensoussan1995}.

By Theorem \ref{thm:ldp.abstract} it suffices to check that the conditions specified in \eqref{eq:ldp.cond.1} and \eqref{eq:ldp.cond.2} are satisfied. The second condition \eqref{eq:ldp.cond.2} holds by Proposition \ref{prop:skeleton.convergence}. Regarding the first condition, let $M > 0$ and let $h_\varepsilon \in \Ac_M$ be such that $h_\varepsilon \to h$ as $\Tc_M$-valued random variables in law. Let
\[
	Y^\varepsilon = \Gc^{\varepsilon} \left( W(\cdot) + \frac{1}{\sqrt{\varepsilon}} \int_0^\cdot h_\varepsilon(s) \, ds \right).
\]
Our goal is to show that $\law(Y^\varepsilon) \to \law(U_h)$, where $U_h$ is the solution of \eqref{eq:skeleton}. By the Girsanov theorem $Y^\varepsilon$ solves the equation
\begin{equation*}
	dY^\varepsilon + \left[ AY^\varepsilon + B(Y^\varepsilon) + F(Y^\varepsilon) \right] \, dt = \sqrt{\varepsilon} \sigma(Y^\varepsilon) \, dW + \sigma(Y^\varepsilon) h_\varepsilon \, dt, \qquad Y^\varepsilon(0) = u.
\end{equation*}
Let $\tau^{Y, \varepsilon}_K$ be the stopping time defined for $K > 0$ by
\[
	\tau^{Y, \varepsilon}_K = \inf \left\lbrace s \geq 0 \mid \int_0^s \Vert Y^\varepsilon \Vert_{H^2}^2 \, dr \geq K \right\rbrace.
\]
An argument similar to the one in \cite[Section 4]{brzezniakslavik} or Section \ref{sect:mdp.preliminary} in this manuscript shows that
\begin{equation}
	\label{eq:ldp.tau.y.convergence}
	\tau^{Y, \varepsilon}_K \to \infty \ \text{as} \ K \to \infty \ \text{in probability uniformly w.r.t.\ } \varepsilon \in (0, \varepsilon_0].
\end{equation}
Let $Z^\varepsilon$ be the solution of
\begin{equation}
	\label{eq:ldp.z.definition}
	dZ^\varepsilon + AZ^\varepsilon \, dt = \sqrt{\varepsilon}\sigma(Y^\varepsilon) \, dW, \qquad Z^\varepsilon(0) = 0.
\end{equation}
By the It\^{o} lemma e.g.\ from \cite[Theorem A.1]{brzezniakslavik} and the bound \eqref{eq:sigma.bnd.V} on $\sigma$ in $L_2(\Uc, V)$ we get
\begin{equation}
	\label{eq:ldp.z.estimate}
	\Eb \left[ \sup_{s \in [0, t \wedge \tau_K^{U, \varepsilon}]} \Vert Z^\varepsilon \Vert^2 + \int_0^{t \wedge \tau_K^{U, \varepsilon}} \vert AZ^\varepsilon \vert^2 \, ds \right] \leq C_{t, M} \varepsilon \Eb\left[ \int_0^{t \wedge \tau_K^{U, \varepsilon}} 1 + \vert AU^\varepsilon \vert^2 \, ds \right].
\end{equation}
By the convergence \eqref{eq:ldp.tau.y.convergence} we have
\[
	Z^\varepsilon \in C\left( [0, t]; V \right) \cap L^2\left(0, t; D(A) \right), \quad \Pb\text{-a.s.\ for all} \ \varepsilon \in (0, \varepsilon_0] \ \text{and} \ K \geq 0.
\]
However, we do not know whether $Z^\varepsilon \in L^2(\Omega; C([0, t], V) \cap L^2(0, t; D(A)))$. Using the estimate \eqref{eq:ldp.z.estimate} and the bounds on $AY^\varepsilon$ from Proposition \ref{prop:ldp.V.estimate} we may employ an argument similar to the one in the proof of the uniform stochastic Gronwall lemma in Proposition \ref{prop:stoch.Gronwall} to show that $Z^\varepsilon \to 0$ in $X$ in probability as $\varepsilon \to 0+$, that is
\begin{equation}
	\label{eq:ldp.z.convergence}
	\lim_{\varepsilon \to 0+} \Pb\left( \left\lbrace \sup_{s \in [0, t]} \Vert Z^\varepsilon \Vert^2 + \int_0^t \vert AZ^\varepsilon \vert^2 \, ds \geq \eta \right\rbrace \right) = 0 \quad \text{for all} \ \eta > 0.
\end{equation}

Let us also define $R^\varepsilon = Y^\varepsilon - Z^\varepsilon$, that is $R^\varepsilon$ solves
\begin{equation}
	\label{eq:ldp.r.definition}
	dR^\varepsilon + \left[AR^\varepsilon + B(R^\varepsilon + Z^\varepsilon) + F(R^\varepsilon + Z^\varepsilon) \right] \, dt = \sigma(R^\varepsilon + Z^\varepsilon) h_\varepsilon \, dt, \qquad R^\varepsilon(0) = u.
\end{equation}

Since $h_\varepsilon \to h$ in law on the space $\Tc_m$ and $Z_\varepsilon \to 0$ in $X$ in probability by \eqref{eq:ldp.z.convergence}, one can use the Portmanteau lemma in Polish spaces, see e.g.\ \cite[Theorem 18.2.6]{fristedtgray}, to establish a version of \cite[Theorem 2.17(v)]{vandervaart}, that is that the convergences above imply $(h_\varepsilon, Z_\varepsilon) \to (h, 0)$ in law on the space $\Tc_M \times X$. By the Skorokhod representation theorem, see e.g.\ \cite[Theorem 2.4]{dpz}, there exists a probability space $( \tilde{\Omega}, \tilde{\Fc}, \tilde{\Pb} )$ and $\Tc_M \times \Tc_M \times X \times C([0, t], \Uc_0) \times C([0, t], \Uc_0)$-valued random variables $(\tilde{h}_\varepsilon, \tilde{h}, \tilde{Z}_\varepsilon, \tilde{W}_\varepsilon, \tilde{W})$ such that $\tilde{W}_\varepsilon$ and $\tilde{W}$ are cylindrical Wiener processes with reproducing kernel Hilbert space $\Uc$ and
\begin{equation}
	\label{eq:ldp.skorokhod}
	\law\left( h_\varepsilon, Z_\varepsilon, W \right) = \law\left( \tilde{h}_\varepsilon, \tilde{Z}_\varepsilon, \tilde{W}_\varepsilon \right), \quad \left( \tilde{Z}_\varepsilon, \tilde{h}_\varepsilon, \tilde{W}_\varepsilon \right) \to \left( 0, \tilde{h}, \tilde{W} \right) \ \Pb\text{-a.s.\ in} \ \left( X, \Tc_M, C\left([0, t], \Uc_0 \right)\right).
\end{equation}

Let us define
\[
	\tilde{Y}^\varepsilon = \Gc^\varepsilon\left( \tilde{W}(\cdot) + \frac{1}{\sqrt{\varepsilon}} \int_0^\cdot \tilde{h}_\varepsilon(s) \, ds \right).
\]
Using the Girsanov theorem again we observe that $\tilde{Y}^\varepsilon$ solves the equation
\[
	d\tilde{Y}^\varepsilon + \left[ A\tilde{Y}^\varepsilon + B(\tilde{Y}^\varepsilon) + F(\tilde{Y}^\varepsilon) \right] \, dt = \sqrt{\varepsilon} \sigma(\tilde{Y}^\varepsilon) \, d\tilde{W}_\varepsilon + \sigma(\tilde{Y}^\varepsilon) \tilde{h}_\varepsilon \, dt, \qquad \tilde{Y}^\varepsilon(0) = u.
\]
By \cite[Theorem 1.6]{qiao2010} we have
\begin{equation}
	\label{eq:ldp.yw.laws}
	\law\left( Y^\varepsilon, W \right)	= \law\left( \tilde{Y}^\varepsilon, \tilde{W}_\varepsilon \right).
\end{equation}

The equality of laws \eqref{eq:ldp.yw.laws} allows us to use a stopped version of the Bensoussan trick to show that $\tilde{Z}^\varepsilon$ solves the equation
\[
	d\tilde{Z}^\varepsilon + A\tilde{Z}^\varepsilon \, dt = \sqrt{\varepsilon}\sigma(\tilde{Y}^\varepsilon) \, d\tilde{W}_\varepsilon, \qquad \tilde{Z}^\varepsilon(0) = 0.
\]
Indeed, defining $\tilde{\Fc}_s^\varepsilon = \sigma( \mathds{1}_{[0, s]} \tilde{W}_\varepsilon, \mathds{1}_{[0, s]} \tilde{Y}^\varepsilon, \mathds{1}_{[0, s]} \tilde{Z}^\varepsilon)$, $\tilde{\Fc}_s = \sigma( \mathds{1}_{[0, s]} \tilde{W}, \mathds{1}_{[0, s]} \tilde{Y})$, $\tilde{\Fb}^\varepsilon = ( \tilde{\Fc}_t^\varepsilon )_{t \geq 0}$ and $\tilde{\Fb} = (\tilde{\Fc}_t)_{t \geq 0}$, we observe that $\tilde{W}_\varepsilon$, resp.\ $\tilde{W}$, is a cylindrical $\tilde{\Fb}^\varepsilon$-Wiener, resp.\ $\tilde{\Fb}$-Wiener, process. Let for $K \geq 0$
\begin{equation}
	\label{eq:tilde.tau.k}
	\tilde{\tau}_K^{Y, \varepsilon} = \inf \left\lbrace t \geq 0 \mid \int_0^t \Vert \tilde{Y}^\varepsilon \Vert_{H^2}^2 \, ds \geq K \right\rbrace. 
\end{equation}
Then since the map that maps a single trajectory in $L^2(0, T; H^2)$ to the hitting time defined in \eqref{eq:tilde.tau.k} is Borel measurable, $\tilde{\tau}_K^{Y, \varepsilon}$ is a random variable. By the debut theorem \cite[Theorem 50, p.\ 116]{dellacheriemeyer} $\tilde{\tau}_K^{Y, \varepsilon}$ is a $\tilde{\Fb}^\varepsilon$-stopping time. It remains to apply the Bensoussan trick from \cite[Section 4.3.4]{bensoussan1995} to the equation for $Z^\varepsilon_K = Z^\varepsilon(\cdot \wedge \tau_K^{Y, \varepsilon})$
\[
	dZ^\varepsilon_K + \mathds{1}_{[0, \tau_K^{Y, \varepsilon}]} AZ^\varepsilon_K \, dt = \mathds{1}_{[0, \tau_K^{Y, \varepsilon}]} \sqrt{\varepsilon}\sigma(Y^\varepsilon) \, dW, \qquad Z^\varepsilon(0) = 0,
\]
to show that $\tilde{Z}_K^\varepsilon = \tilde{Z}^\varepsilon(\cdot \wedge \tilde{\tau}^{Y, \varepsilon}_K)$ is a solution of
\[
	d\tilde{Z}^\varepsilon_K + \mathds{1}_{[0, \tilde{\tau}_K^{Y, \varepsilon}]} A\tilde{Z}^\varepsilon_K \, dt = \mathds{1}_{[0, \tilde{\tau}_K^{Y, \varepsilon}]} \sqrt{\varepsilon}\sigma(\tilde{Y}^\varepsilon) \, d\tilde{W}_\varepsilon, \qquad \tilde{Z}^\varepsilon(0) = 0,
\]
and take the limit $K \to \infty$.

Defining $\tilde{R}^\varepsilon = \tilde{Y}^\varepsilon - \tilde{Z}^\varepsilon$ we now observe that $\tilde{R}^\varepsilon$ satisfies
\[
	d\tilde{R}^\varepsilon + \left[A\tilde{R}^\varepsilon + B(\tilde{R}^\varepsilon + \tilde{Z}^\varepsilon) + F(\tilde{R}^\varepsilon + \tilde{Z}^\varepsilon) \right] \, dt = \sigma(\tilde{R}^\varepsilon + \tilde{Z}^\varepsilon) \tilde{h}_\varepsilon \, dt, \qquad \tilde{R}^\varepsilon(0) = u.
\]
Referring to the equality in laws in \eqref{eq:ldp.skorokhod} and \eqref{eq:ldp.yw.laws} we observe that
\begin{equation}
	\label{eq:ldp.r.laws}
	\law(\tilde{R}^\varepsilon) = \law(R^\varepsilon).
\end{equation}
By the $\tilde{\Pb}$-a.s.\ convergences in \eqref{eq:ldp.skorokhod} we deduce that $\tilde{R}^\varepsilon \to \tilde{R}^0$ $\tilde{\Pb}$-almost surely, where $\tilde{R}^0$ satisfies
\[
	d\tilde{R}^0 + \left[ A\tilde{R}^0 + B(\tilde{R}^0) + F(\tilde{R}^0) \right] \, dt = \sigma(\tilde{R}^0) \tilde{h} \, dt, \qquad \tilde{R}^0(0) = u,
\]
and therefore, recalling \eqref{eq:ldp.skorokhod} again, we have
\begin{equation}
	\label{eq:ldp.R.U.law}
	\law(\hat{R}^0) = \law(U_{\tilde{h}}) = \law(U_h).
\end{equation}

We are now ready to finalize the proof. By \eqref{eq:ldp.r.laws} we have $\law(Y^\varepsilon - Z_\varepsilon) = \law(\tilde{R}^\varepsilon)$. Using the convergence $\tilde{R}^\varepsilon \to \tilde{R}^0$ $\tilde{\Pb}$-almost surely and \eqref{eq:ldp.R.U.law} we have $\law(Y^\varepsilon - Z^\varepsilon) \to \law(U_h)$. Recalling $Z^\varepsilon \to 0$ in probability, we finally have $\law(Y^\varepsilon) \to \law(U_h)$. By Theorem \ref{thm:ldp.abstract} the LDP holds with the good rate function $I: C([0, t], V) \cap L^2(0, t; D(A)) \to \Rb$ given by
\begin{equation}
	\label{eq:ldp.rate.actual}
	I(U) = \inf \left\lbrace \frac12 \int_0^t \vert h \vert_{\Uc}^2 \, ds \mid h \in L^2(0, t; \Uc) \ \text{s.t.} \ U = \Gc^0\left(\int_0^\cdot h \, ds\right) \right\rbrace
\end{equation}
where $\Gc^0$ has been defined in \eqref{eq:gc.0}. \qed

\section{Moderate deviations principle}
\label{sect:mdp}

Let $u_0 \in V \cap H^2$ and let $U^0$ be the solution of the deterministic equation
\begin{equation}
	\label{eq:mdp.u.zero}
	dU^0 + \left[ AU^0 + B\left( U^0 \right) + F\left( U^0 \right) \right] \, dt = 0, \qquad U^0(0) = u_0.
\end{equation}
Let $R^\varepsilon$ be defined by
\begin{equation}
	\label{eq:r.varep.definition}
	R^\varepsilon = \frac{U^\varepsilon - U^0}{\sqrt{\varepsilon}\lambda(\varepsilon)}, \qquad \varepsilon \in (0, \varepsilon_0],
\end{equation}
where $U^\varepsilon$ is the solution of \eqref{eq:ldp.main} with the same initial condition $u_0$. It is straightforward to check that $R^\varepsilon$ satisfies the equation
\begin{multline}
	\label{eq:r.varep}
	dR^\varepsilon + \left[ AR^\varepsilon + B\left( R^\varepsilon, U^0 + \sqrt{\varepsilon} \lambda(\varepsilon) R^\varepsilon \right) + B\left(U^0, R^\varepsilon\right) + \Apr R^\varepsilon + E R^\varepsilon \right] \, dt\\
	= \lambda^{-1}(\varepsilon)\sigma\left(U^0 + \sqrt{\varepsilon} \lambda(\varepsilon) R^\varepsilon\right) \, dW, \quad R^\varepsilon(0) = 0.
\end{multline}
The existence of $R^\varepsilon$ follows directly from the definition \eqref{eq:r.varep.definition}. Similarly as for the LDP we need to establish certain estimates on $R^\varepsilon$ before we proceed to the proof of the main theorem. The process itself is more technically involved than the respective estimates for $U^\varepsilon$ and seems to require additional regularity of the deterministic solution $U^0$. The additional regularity will be discussed in detail in Section \ref{sect:mdp.preliminary}.

\subsection{Skeleton equation}
\label{sect:mdp.skeleton}

Let $h \in L^2(0, t; \Uc)$. The skeleton equation corresponding to \eqref{eq:r.varep} is the equation
\begin{multline}
	\label{eq:mdp.skeleton}
	dR_h + \left[ AR_h + B\left( R_h, U^0 \right) + B\left(U^0, R_h \right) + \Apr R_h + E R_h \right] \, dt\\
	= \sigma\left( U^0 \right) h \, dt, \quad R^\varepsilon(0) = 0.
\end{multline}

Following a similar argument as in Section \ref{sect:mdp.skeleton} we may establish the following equivalent of Proposition \ref{prop:skeleton.convergence}. The additional regularity of $R_h$ is not required in this case since the potentially problematic term with $\sigma$ does not play any role in the compactness argument.

\begin{proposition}
\label{prop:mdp.skeleton}
Let $\lbrace h_n \rbrace_{n=1}^\infty \subseteq \Tc_M$ and let $R^n = R_{h_n}$ be the respective solutions of \eqref{eq:mdp.skeleton}. Then there exist $h \in L^2(0, t; \Uc)$ and a subsequence (for simplicity not relabelled) such that $h_n \harp h$ in $L^2(0, t; \Uc)$ and $R^n \to R_h$ in $L^\infty(0, t; V) \cap L^2(0, t; D(A))$.

Moreover, there exists a measurable map $\Gc^0_R: C([0, t], \Uc_0) \to C([0, t], V) \cap L^2(0, t; D(A))$ such that $R_h = \Gc^0_R(\int_0^\cdot h_s \, ds)$.
\end{proposition}

\subsection{Preliminary estimates}
\label{sect:mdp.preliminary}

In the estimates in this subsection we require additional regularity of the solution $U^0$ of the deterministic equation \eqref{eq:mdp.u.zero}. Even though this is restrictive, it is not unusual to assume more regular initial data to obtain qualitative results on the primitive equations, see e.g.\ the mathematical justification of the hydrostatic limit in \cite{li2019}. To be more precise, it suffices to assume
\[
	U^0 \in C\left( [0, t], H^1 \right) \cap L^2 \left(0, t; H^2 \right) \cap L^\infty \left( 0, t; L^\infty \right), \quad \partial_z U^0 \in L^\infty\left( 0, t; H^1 \right) \cap L^2\left( 0, t; H^2 \right).
\]
Assuming $u_0 \in V \cap H^2$, by \cite[Proposition 2.1]{caolititi2014} there exists a unique global solution such that
\[
	U^0 \in C\left([0, t], H^2 \right) \cap L^2\left( 0, t; H^3 \right) \ \text{for all} \ t \geq 0.
\]
Let $\tau_K^0$ be the stopping time defined by
\begin{equation*}
	\begin{split}
	\tau_K^0 = \inf \Bigg\lbrace s \geq 0 \mid &\sup_{r \in [0, s]} \Vert U^0 \Vert + \sup_{r \in [0, s]} \vert U^0 \vert_{L^\infty} + \sup_{r \in [0, s]} \Vert \partial_z U^0 \Vert\\
	&+\int_{0}^s \Vert U^0 \Vert_{H^2}^2 + \Vert \partial_z U^0 \Vert_{H^2}^2 \, dr \geq K \Bigg\rbrace.
	\end{split}
\end{equation*}
Since $\tau_K^0$ is essentially deterministic and $U^0$ is a global solution, we observe
\begin{equation}
	\label{eq:mpd.tau.0.convergence}
	\lim_{K \to \infty} \Pb\left( \left\lbrace \tau_K^0 \leq \tilde{t} \right\rbrace \right) = 0 \ \text{for all} \ \tilde{t} \geq 0.
\end{equation}

\begin{proposition}
\label{prop:mdp.weak}
Let $\tau_K^{w, \varepsilon}$ be the stopping time defined for $K \geq 0$ and $\varepsilon \in (0, \varepsilon_0]$ by
\begin{equation*}
	\tau_K^{w, \varepsilon} = \inf \left\lbrace s \geq 0 \mid \int_0^s \vert \Re \vert^4 + \Vert \Re \Vert^2 + \vert \Re \vert^4 \Vert \Re \Vert^2\, dr \geq K \right\rbrace.
\end{equation*}
Then $\tau_K^{w, \varepsilon} \to \infty$ $\Pb$-almost surely as $K \to \infty$ for all $\varepsilon \in (0, \varepsilon_0]$. Moreover for all $\tilde{t} \geq 0$ one has
\[
	\lim_{K \to \infty} \Pb \left( \left\lbrace \tau_K^{w, \varepsilon} \leq \tilde{t} \right\rbrace \right) = 0 \ \text{uniformly w.r.t.\ } \varepsilon \in (0, \varepsilon_0].
\]
\end{proposition}

\begin{proof}
Let $p \geq 2$. Applying the It\^{o} lemma \cite[Theorem A.1]{brzezniakslavik} we obtain
\begin{multline*}
	d\vert \Re \vert^p + p \vert \Re \vert^{p-2} \left[ \Vert \Re \Vert^2 + \left( B\left(\Re, U^0\right), \Re \right) + \left( \Apr \Re + E \Re, \Re \right) \right] \, dt\\
	\leq p \lambda^{-1}(\varepsilon) \vert \Re \vert^{p-2} \left( \sigma(U^0 + \sqrt{\varepsilon} \lambda(\varepsilon) \Re) \, dW, \Re \right) + p \vert \Re \vert^p \, dt\\
	+ \frac{p(p-1)}{2} \lambda^{-2}(\varepsilon) \vert \Re \vert^{p-2} \Vert \sigma(U^0 + \sqrt{\varepsilon} \lambda(\varepsilon) \Re) \Vert_{L_2(\Uc, H)}^2 \, dt.
\end{multline*}
By the estimate \eqref{eq:b.estimate3} on $B$ we have
\begin{align*}
	p \vert \Re \vert^{p-2} \left| \left( B\left(\Re, U^0\right), \Re \right) \right| &\leq Cp \vert \Re \vert^{p-2} \Vert \Re \Vert \Vert U^0 \Vert^{1/2} \Vert U^0 \Vert_{H^2}^{1/2} \vert \Re \vert^{1/2} \Vert \Re \Vert^{1/2}\\
	&\leq \eta \vert \Re \vert^{p-2} \Vert \Re \Vert^2 + C_\eta \vert \Re \vert^p \Vert U^0 \Vert^2 \Vert U^0 \Vert_{H^2}^2
\end{align*}
for some $\eta > 0$ precisely determined later. Using the boundedness of the operators $\Apr$ and $E$ we get
\begin{equation*}
	p \vert \Re \vert^{p-2} \left| \left( \Apr \Re + E \Re, \Re \right) \right| \leq C \vert \Re \vert^{p-1} \Vert \Re \Vert \leq \eta \vert \Re \vert^{p-2} \Vert \Re \Vert^2 + C_\eta \vert \Re \vert^p.
\end{equation*}
By the bound \eqref{eq:sigma.bnd.H} on $\sigma$ in $L_2(\Uc, H)$ we have
\begin{multline*}
	\frac{p(p-1)}{2} \lambda^{-2}(\varepsilon) \vert \Re \vert^{p-2} \Vert \sigma(U^0 + \sqrt{\varepsilon} \lambda(\varepsilon) \Re) \Vert_{L_2(\Uc, H)}^2\\
	\leq C \lambda^{-2}(\varepsilon) \vert \Re \vert^{p-2} \left( 1 + \vert U^0 \vert^2 + \Vert U^0 \Vert^2 \right) + C \varepsilon \vert \Re \vert^p + p(p-1) \varepsilon \eta_0 \vert \Re \vert^{p-2} \vert \grad_3 \Re \vert^2.
\end{multline*}
Let $K \geq 0$, $N \in \Nb$ and let $\tau_a$ and $\tau_b$ be stopping times such that $0 \leq \tau_a \leq \tau_b \leq t \wedge \tau_K^0$. We estimate the stochastic integral using the Burkholder-Davis-Gundy inequality \eqref{eq:bdg} and the bound \eqref{eq:sigma.bnd.H} on $\sigma$ in $L_2(\Uc, H)$ as
\begin{align*}
	p \lambda^{-1}(\varepsilon) &\Eb \sup_{s \in [\tau_a, \tau_b]} \left| \int_{\tau_a}^s \vert \Re \vert^{p-2} \left( \sigma(U^0 + \sqrt{\varepsilon} \lambda(\varepsilon) \Re) \, dW, \Re \right) \right|\\
	&\leq C_{BDG} p \lambda^{-1} \Eb \left( \int_{\tau_a}^{\tau_b} \vert \Re \vert^{2p-2} \Vert \sigma(U^0 + \sqrt{\varepsilon} \lambda(\varepsilon) \Re) \Vert_{L_2(\Uc, H)}^2 \, dr \right)^{1/2}\\
	&\leq C \lambda^{-1}(\varepsilon) \Eb \left( \int_{\tau_a}^{\tau_b}  \vert \Re \vert^{2p-2} \left( 1 + \Vert U^0 \Vert^2 + \varepsilon \lambda^{2}(\varepsilon) \vert \Re \vert^2 \right) \, dr \right)^{1/2}\\
	&\hphantom{\leq \ } + C_{BDG} p \sqrt{2 \varepsilon \eta_0} \Eb \left( \int_{\tau_a}^{\tau_b} \vert \Re \vert^{2p-2} \Vert \Re \Vert^2 \, dr \right)^{1/2}\\
	&= I_1 + I_2.
\end{align*}
Using the Young inequality we obtain
\begin{align}
	\nonumber
	I_1 &\leq C \lambda^{-1}(\varepsilon) \Eb \left[ \sup_{s \in [\tau_a, \tau_b]} \vert \Re \vert^{p/2} \left( \int_{\tau_a}^{\tau_b} \vert \Re \vert^{p-2} \left( 1 + \Vert U^0 \Vert^2  + \varepsilon \lambda^2(\varepsilon) \vert \Re \vert^2 \right) \, dr \right)^{1/2} \right]\\
	\label{eq:w.i1}
	&\leq \eta \Eb \sup_{s \in [\tau_a, \tau_b]} \vert \Re \vert^p + C_\eta \lambda^{-2}(\varepsilon) \Eb \int_{\tau_a}^{\tau_b} \vert \Re \vert^{p-2} \left( 1 + \Vert U^0 \Vert^2  \right) \, dr + C_\eta \varepsilon \int_{\tau_a}^{\tau_b} \vert \Re \vert^p  \, dr,\\
	\nonumber
	I_2 &\leq C_{BDG} p \sqrt{2 \varepsilon \eta_0} \Eb \left[ \sup_{s \in [\tau_a, \tau_b]} \vert \Re \vert^{p/2} \left( \int_{\tau_a}^{\tau_b} \vert \Re \vert^{p-2} \Vert \Re \Vert^2 \, dr \right)^{1/2} \right]\\
	\label{eq:w.i2}
	&\leq (1-\delta) \Eb \sup_{s \in [\tau_a, \tau_b]} \vert \Re \vert^p + \frac{C^2_{BDG} p^2 \varepsilon \eta_0}{2(1-\delta)} \Eb \int_{\tau_a}^{\tau_b} \vert \Re \vert^{p-2} \Vert \Re \Vert^2 \, dr,
\end{align}
for some $\delta \in (0, 1)$. Collecting the estimates above, choosing $\delta$ and $\eta$ sufficiently small and assuming that $\eta_0$ or $\varepsilon_0$ are sufficiently small we observe
\begin{equation*}
	\Eb \left[ \sup_{s \in [\tau_a, \tau_b]} \vert \Re \vert^p + \int_{\tau_a}^{\tau_b} \vert \Re \vert^{p-2} \Vert \Re \Vert^2 \, dr \right] \leq C \Eb \int_{\tau_a}^{\tau_b} \left( 1 + \vert \Re \vert^p \right) \left( 1 + \Vert U^0 \Vert^2 \right) \, dr 
\end{equation*}
with the constant $C$ independent of $\varepsilon$. The claim then follows by the means of the uniform stochastic Gronwall lemma from Proposition \ref{prop:stoch.Gronwall}.
\end{proof}

In the rest of this Section let us denote the velocity and temperature part of $\Re$ by $\Re = (\upe, \Upe)$. Following the argument by Cao and Titi \cite{cao2007} we decompose the velocity part $\upe$ into its barotropic and baroclinic modes $\upe = \upeb + \upet$, where
\[
	\upeb = \Ac_2 \upe, \qquad \upet = \Rc \upe.
\]
Similarly as in \cite{cao2007} one can establish that the barotropic mode $\upeb$ satisfies the equation in $\Mc_0$
\begin{equation}
\label{eq:upeb}
	\begin{split}
	d\upeb &+ \bigg[ -\mu \laplace \upeb + \sqrt{\varepsilon} \lambda(\varepsilon) \left[ \left( \upeb \cdot \grad \right) \upeb + \Jc_1(\upe, \upe) \right]	+ \left( \upeb \cdot \grad \right) \vb^0 + \Jc_1(\upe, v^0)\\
	&\hphantom{+ \bigg[ \ } + \left( \vb^0 \cdot \grad \right) \upeb + \Jc_1(v^0, \upe) + \grad p_S^\varepsilon - \beta_T g \grad \Ac_2 \int_{z}^0 \Upe \, dz' + f \vec{k} \times \upeb \bigg] \, dt\\
	&= \lambda^{-1}(\varepsilon) \Ac_2 \sigma_1(U^0 + \sqrt{\varepsilon}\lambda(\varepsilon) \Re) \, dW_1,
	\end{split}
\end{equation}
where for sufficiently regular $u, v$ we denote
\begin{gather*}
	B_1(u, v) = \left(u \cdot \grad \right) v,\\
	\Jc_1(u, v) = \Ac_2\left[ \left( \tilde{u} \cdot \grad \right) \tilde{v} + \left( \div \tilde{u} \right) \tilde{v} \right],
\end{gather*}
with the periodic boundary conditions on $\Mc_0$ and
\[
	\div \upeb = 0, \qquad \upeb(0) = 0.
\]
The baroclinic mode $\upet$ solves the equation in $\Mc_0$
\begin{equation}
\label{eq:upet}
	\begin{split}
	d\upet &+ \bigg[ - \mu \laplace \upet - \nu \partial_{zz} \upet + \sqrt{\varepsilon} \lambda(\varepsilon) \left[ B_1(\upet, \upet) + \Jc_2(\upe, \upe) \right] + B_1(\upet, \vt^0)\\
	&\hphantom{+ \bigg[ \ } + \Jc_2(\upe, v^0) + B_1(\vt^0, \upet) + \Jc_2(v^0, \upe) - \beta_T g \grad \Rc \int^0_{z} \Upe \, dz' + f \vec{k} \times \upet \bigg] \, dt\\
	&= \lambda^{-1}(\varepsilon) \Rc \sigma_1(U^0 + \sqrt{\varepsilon} \lambda(\varepsilon) \Re) \, dW_1,
	\end{split}
\end{equation}
where for sufficiently regular $u, v$ we set
\begin{gather*}
	B_2(u, v) = B_1(u, v) + w(u) \partial_z v = \left( u \cdot \grad \right) v - \left( \int_{-h}^z \div u \, dz' \right) \partial_z v,\\
	\Jc_2(u, v) = \left( \tilde{u} \cdot \grad \right) \bar{v} -  \Jc_1(u, v) = \left( \tilde{u} \cdot \grad \right) \bar{v} - \Ac_3\left[ \left( \tilde{u} \cdot \grad \right) \tilde{v} + \left( \div \tilde{u} \right) \tilde{v} \right]
\end{gather*}
with the periodic boundary conditions on $\Gamma_l$ and
\[
	\partial_z \upet = 0 \ \text{on} \ \Gamma_i \cup \Gamma_b, \qquad \upet(0) = 0.
\]

\begin{proposition}
\label{prop:mdp.l6}
Let $K \geq 0$ and let $\tau_K^{6, \varepsilon}$ be the stopping time defined by
\begin{equation}
	\tau_K^{6, \varepsilon} = \inf \left\lbrace s \geq 0 \mid \int_0^s \vert \upet \vert^6_{L^6} +  \left( \int_{\Mc} \vert \upet \vert^4 \vert \grad_3 \upet \vert^2 \, d\Mc \right) \, dr \geq K \right\rbrace.
\end{equation}
Then $\tau_K^{6, \varepsilon} \to \infty$ $\Pb$-almost surely as $K \to \infty$ for all $\varepsilon \in (0, \varepsilon_0]$. Moreover for all $\tilde{t} \geq 0$ one has
\[
	\lim_{K \to \infty} \Pb \left( \left\lbrace \tau_K^{6, \varepsilon} \leq \tilde{t} \right\rbrace \right) = 0 \ \text{uniformly w.r.t.\ } \varepsilon \in (0, \varepsilon_0].
\]
\end{proposition}

\begin{proof}
We will need the following estimates from \cite[p 256]{cao2007}. There exists $C > 0$ such that for all $f \in L^6$ satisfying $\grad_3 (f^3) \in L^2$ one has
\begin{align}
	\label{eq:ct63}
	\vert f \vert_{L^{12}_x L^6_z}^6 &\leq C \vert f \vert_{L^6}^3 \vert \grad_3 f^3 \vert_{L^2} + \vert f \vert_{L^6}^6,\\
	\label{eq:ct64}
	\vert f \vert_{L^8_x L^4_z}^4 &\leq C \vert f \vert^3_{L^6} \Vert f \Vert,\\
	\label{eq:ct65}
	\vert f \vert_{L^8_x L^2_z}^2 &\leq C \vert f \vert_{L^6}^{3/2} \Vert f \Vert^{1/2}.
\end{align}
The estimate \eqref{eq:ct63} follows directly from \eqref{eq:anis.1} with $q=4$, the others can be established using the Minkowski inequality similarly as in Lemma \ref{lemma:anisotropic.estimates}. 

Employing the It\^{o} lemma \cite[Theorem A.1]{brzezniakslavik} and using integration by parts and the cancellation property 
\begin{equation*}
	\left( B_2(u, v), \vert v \vert^q v \right) = 0, \qquad u \in V_1, v \in V_1 \cap H^2,
\end{equation*}
where $q \geq 0$ and $\vert v \vert$ denotes the modulus of $v$, we have
\begin{align*}
	d\vert \upet \vert^6_{L^6} &+ 6 \left( \mu \wedge \nu \right) \int_{\Mc} \vert \upet \vert^4 \vert \grad_3 \upet \vert^2 \, d\Mc \, dt\\
	&\leq 6 \left| \int_{\Mc} \vert \upet \vert^4 \upet \cdot \left( \beta_T g \grad \Rc \int_{z}^0 \Upe \, dz' + \sqrt{\varepsilon} \lambda(\varepsilon) \Jc_2(\upe, \upe) \right) \, d\Mc \right| \, dt\\
	&\hphantom{\leq \ } + 6 \left| \int_{\Mc} \vert \upet \vert^4 \upet \cdot \left( B_2(\upet, \vt^0) + \Jc_2(\upe, v^0) + \Jc_2(v^0, \upe)\right) \, d\Mc \right| \, dt\\
	&\hphantom{\leq \ } +6 \lambda^{-1} \sum_{k=1}^\infty \int_{\Mc} \vert \upet \vert^4 \upet \cdot \left( \Rc \sigma_1(U^0 + \sqrt{\varepsilon}\lambda(\varepsilon) R^\varepsilon) e_k \right) \, d\Mc \, dW_1^k\\
	&\hphantom{\leq \ } + 15 \lambda^{-2}(\varepsilon) \sum_{k=1}^\infty \int_{\Mc} \vert \upet \vert^4 \left| \Rc \sigma_1(U^0 + \sqrt{\varepsilon}\lambda(\varepsilon) R^\varepsilon) \right|^2 \, d\Mc \, dt\\
	&= I_1 \, dt + I_2 \, dt + \sum_{k = 1}^\infty I_3^k \, dW_1^k + \sum_{k = 1}^\infty I_4^k \, dt.
\end{align*}

The estimates of $I_1$ follow from \cite[Section 3.2]{cao2007}. Estimates of the stochastic terms $I_3^k$ and the It\^{o} correction terms $I^k_4$ are essentially the same as in \cite[Proposition 4.3]{brzezniakslavik}. The remaining term $I_2$ seems to require more work. We use and slightly refine ideas from \cite[Proposition 4.3]{debussche2012} with the anisotropic estimates from \cite{guillengonzalez2001}. 

Before we proceed to the estimates of the integrals above, let us make a short auxiliary estimate. Let $u, v, w \in D(A)$. Following the argument from \cite[Section 3.2]{cao2007} one may use integration by parts to establish
\[
	-\int_{\Mc} \vert \vt \vert^4 \vt \cdot \Jc_2(u, w) \, d\Mc = \int_{\Mc} \left( \div \tilde{u} \right) \vert \vt \vert^4 \vt \cdot \bar{w} + \left( \tilde{u} \cdot \grad \right) \left( \vert \vt \vert^4 \vt \right) \cdot \bar{w} - \partial_j \left( \vert \vt \vert^4 \vt \right) \overline{\tilde{u}_j \tilde{w}_k} \, d\Mc,
\]
which in turn leads to the estimate
\begin{align}
	\nonumber
	\bigg| \int_{\Mc} \vert \vt \vert^4 \vt &\cdot \Jc_2(u, w) \, d\Mc \bigg|\\
	\nonumber
	&\leq \int_{\Mc_0} \vert \bar{w} \vert \left( \int_{-h}^0 \vert \grad \tilde{u} \vert \vert \vt \vert^5 \, dz \right) \, d\Mc_0 + \int_{\Mc_0} \vert \bar{w} \vert \left( \int_{-h}^0 \vert \tilde{u} \vert \vert \grad \vt \vert \vert \vt \vert^4 \, dz \right) \, d\Mc_0\\
	\nonumber
	&\hphantom{\leq \ } + \int_{\Mc_0} \left( \int_{-h}^0 \vert \tilde{u} \vert \vert \tilde{w} \vert \, dz \right) \left( \int_{-h}^0 \vert \grad \vt \vert \vert \vt \vert^4 \, dz \right) \, d\Mc_0\\
	\label{eq:j2split}
	&= \Jc_2^1(v, u, w) + \Jc_2^2(v, u, w) + \Jc_2^3(v, u, w).
\end{align}

We split the integral $I_1$ into two terms as
\begin{align*}
	I_1 &\leq 6 \beta_T g \left| \int_{\Mc} \vert \upet \vert^4 \upet \cdot \grad \Rc \left( \int_{z}^0 \Upe \, dz' \right) \, d\Mc \right| + 6 \sqrt{\varepsilon} \lambda(\varepsilon) \left| \int_{\Mc} \vert \upet \vert^4 \upet \cdot \Jc_2(\upe, \upe) \, d\Mc \right|\\
	&= I_1^1 + I_1^2.
\end{align*}
Regarding $I_1^1$ we use integration by parts, the Ladyzhenskaya inequality and the estimate \eqref{eq:ct64} to obtain
\begin{equation}
	\label{eq:l6.i1.1}
	\begin{split}
	I_1^1 &= 6 \beta_T g \left| \int_{\Mc} \div\left( \vert \upet \vert^4 \upet \right) \Rc \left( \int_z^0 \Upe \, dz' \right) \, d\Mc \right|\\
	&\leq C \int_{\Mc_0} \overline{\vert \Upe \vert} \left( \int_{-h}^0 \vert \grad \upet \vert \vert \upet \vert^4 \, dz \right) \, d\Mc_0\\
	&\leq C \int_{\Mc_0} \overline{\vert \Upe \vert} \vert \grad (\upet)^3 \vert_{L^2_z} \vert \upet \vert_{L^4_z}^2 \, d\Mc_0\\
	&\leq C \vert \overline{\vert \Upe \vert} \vert_{L^4(\Mc_0)} \vert \grad_3 (\upet)^3 \vert \vert \upet \vert_{L^8_x L^4_z}^2\\
	&\leq C \left( \vert \overline{\vert \Upe \vert} \vert_{L^2(\Mc_0)}^{1/2} \vert \grad \overline{\vert \Upe \vert} \vert_{L^2(\Mc_0)}^{1/2} + \vert \overline{\vert \Upe \vert} \vert_{L^2(\Mc_0)} \right) \vert \grad_3 (\upet)^3 \vert \vert \upet \vert_{L^6}^{3/2} \Vert \upet \Vert^{1/2}\\
	&\leq \eta \vert \grad_3 (\upet)^3 \vert^2 + C_\eta \vert \upet \vert_{L^6}^6 \left( \vert \Re \vert^2 \Vert \Re \Vert^2 + \vert \Re \vert^4 \right) + C_\eta \Vert \Re \Vert^2.
	\end{split}
\end{equation}
We split the integral $I_1^2$ using \eqref{eq:j2split}. The first term can be estimated by the estimate \eqref{eq:ct63} and the Ladyzhenskaya inequality by
\begin{align*}
	\Jc_2^1(\upe, \upe, \upe) &\leq \int_{\Mc_0}\overline{\vert \upet \vert} \vert \grad_3 (\upet)^3 \vert_{L^2_z} \vert \upet \vert_{L^6_z}^3 \, d\Mc_0\\
	&\leq \vert \overline{\vert \upet \vert} \vert_{L^4(\Mc_0)} \vert \grad_3 (\upet)^3 \vert \vert \upet \vert_{L^{12}_x L^6_z}^3\\
	&\leq C \left( \vert \overline{\vert \upet \vert} \vert_{L^2(\Mc_0)}^{1/2} \vert \grad \overline{\vert \upet \vert} \vert_{L^2(\Mc_0)}^{1/2} + \vert \overline{\vert \upet \vert} \vert_{L^2(\Mc_0)} \right) \vert \grad_3 (\upet)^3 \vert\\
	&\hphantom{\leq \ }\cdot \left( \vert \upet \vert_{L^6}^{3/2} \vert \grad_3 (\upet)^3 \vert^{1/2} + \vert \upet \vert_{L^6}^3 \right)
\end{align*}
and therefore by the Young inequality we get
\begin{multline}
	6\sqrt{\varepsilon} \lambda(\varepsilon) \Jc_2^1(\upe, \upe, \upe)\\
	\leq \eta \vert \grad_3 (\upet)^3 \vert^2 + C_{\eta, \varepsilon} \vert \upet \vert_{L^6}^6 \left( \vert \Re \vert^2 \Vert \Re \Vert^2 + \vert \Re \vert \Vert \Re \Vert + \vert \Re \vert^4 + \vert \Re \vert^2 \right),
\end{multline}
where for fixed $\eta$ we have $C_{\eta, \varepsilon} \to 0$ as $\varepsilon \to 0+$. The second term in the splitting of $I^2_1$ is estimated in exactly the same way as the first one and therefore we have
\begin{equation}
	\Jc_2^2(\upe, \upe, \upe) \leq \eta \vert \grad_3 (\upet)^3 \vert^2 + C_{\eta, \varepsilon} \vert \upet \vert_{L^6}^6 \left( \vert \Re \vert^2 \Vert \Re \Vert^2 + \vert \Re \vert \Vert \Re \Vert + \vert \Re \vert^4 + \vert \Re \vert^2 \right)
\end{equation}
with $C_{\eta, \varepsilon}$ as above. The third term in the splitting of $I^2_1$ is can be bounded using the estimates \eqref{eq:ct64} and \eqref{eq:ct65} as
\begin{align*}
	\Jc_2^3(\upe, \upe, \upe) &\leq \int_{\Mc_0} \vert \upet \vert_{L^2_z}^2 \vert \grad_3 (\upet)^3 \vert_{L^2_z} \vert \upet \vert_{L^2_{4}}^2 \, d\Mc_0\\
	&\leq \vert \upet \vert_{L^8_x L^2_z}^2 \vert \grad_3 (\upet)^3 \vert \vert \upet \vert_{L^8_x L^4_z}^2\\
	&\leq C \vert \upet \vert_{L^6}^{3/2} \Vert \upet \Vert^{1/2} \vert \grad_3 (\upet)^3 \vert \vert \upet \vert_{L^6}^{3/2} \Vert \upet \Vert^{1/2}.
\end{align*}
Employing the Young inequality we get
\begin{equation}
	\label{eq:l6.i1.23}
	6\sqrt{\varepsilon} \lambda(\varepsilon) \Jc_2^1(\upe, \upe, \upe)\\
	\leq \eta \vert \grad_3 (\upet)^3 \vert^2 + C_{\eta, \varepsilon} \vert \upet \vert_{L^6}^6 \Vert \Re \Vert^2,
\end{equation}
where for fixed $\eta$ we have $C_{\eta, \varepsilon} \to 0$ as $\varepsilon \to 0+$. Collecting the estimates \eqref{eq:l6.i1.1}--\eqref{eq:l6.i1.23} and applying the Young inequality we get
\begin{equation}
	\label{eq:l6.i1}
	I_1 \leq \eta \vert \grad_3 (\upet)^3 \vert^2 + C_{\eta} \vert \upet \vert_{L^6}^6 \left( \vert \Re \vert^2 \Vert \Re \Vert^2 + \vert \Re \vert^4 + \Vert \Re \Vert^2 \right) + C_\eta \Vert \Re \Vert^2.
\end{equation}

Let us proceed to $I_2$. We split the integral as
\begin{align*}
	I_2 &\leq 6 \left| \int_{\Mc} \vert \upet \vert^4 \upet \cdot B_2(\upet, \vt^0) \, \Mc \right| + 6 \left| \int_{\Mc} \vert \upet \vert^4 \upet \cdot \Jc_2(\upe, v^0) \, \Mc \right|\\
	&\hphantom{\leq\ } + 6 \left| \int_{\Mc} \vert \upet \vert^4 \upet \cdot \Jc_2(v^0, \upe) \, \Mc \right|\\
	&= I_2^1 + I_2^2 + I_2^3.
\end{align*}
Let us estimate the term $I^1_2$ by
\[
	I^1_2 \leq 6 \left| \int_{\Mc} \vert \upet \vert^4 \upet \cdot \left[ \left( \upet \cdot \grad \right) v^0 \right] \, \Mc \right| + 6 \left| \int_{\Mc} \vert \upet \vert^4 \upet \cdot \left[ w(\upet) \partial_z v^0 \right] \, \Mc \right| = I_2^{11} + I_2^{12}.
\] 
We estimate $I_2^{11}$ using the Gagliardo-Nirenberg inequality as
\begin{equation}
	\label{eq:l6.i2.11}
	\begin{split}
			I_2^{11} &\leq C \int_{\Mc} \vert \grad v^0 \vert \vert \upet \vert^6 \, d\Mc \leq C \vert \grad v^0 \vert_{L^6} \vert \upet \vert_{L^{36/5}}^6 = C \vert \grad v^0 \vert_{L^6} \vert (\upet)^3 \vert_{L^{12/5}}^2\\
			&\leq C \vert \grad v^0 \vert_{L^6} \left( \vert \upet \vert_{L^6}^{3/2} \vert \grad_3 (\upet)^3 \vert^{1/2} + \vert \upet \vert_{L^6}^2\right)\\
			&\leq \eta \vert \grad_3 (\upet)^3 \vert^2 + C_\eta \vert \upet \vert_{L^6}^2 \left( \Vert v^0 \Vert_{H^2}^{4/3} + \Vert v^0 \Vert_{H^2} \right).
	\end{split}
\end{equation}
Regarding $I_2^{12}$ we proceed similarly as in the proof of \cite[Proposition 4.3]{debussche2012} and use integration by parts to obtain
\begin{align*}
	-\int_{\Mc} &\left( \int_{-h}^z \partial_j \upet_j \, dz' \right) \partial_z v^0_k (\upet_k)^5 \, d\Mc\\
	&=  \int_{\Mc} \left( \int_{-h}^z \upet_j \, dz' \right) \partial_j \partial_z v^0_k( \upet_k)^5 + 5 \int_{\Mc} \left( \int_{-h}^z \upet_j \, dz' \right) \partial_z v^0_k \partial_j \upet_k (\upet_k)^4 \, d\Mc\\
	&= I_2^{121} + I_2^{122}.
\end{align*}
Using the anisotropic estimate \eqref{eq:anis.4} and the embedding $L^q \hook L^q_xL^2_z$ with $q = 6$ we get
\begin{equation}
	\label{eq:l6.i2.121}
	\begin{split}
		I_2^{121} &\leq C \int_{\Mc_0} \left( \int_{-h}^0 \vert \upet_j \vert \, dz' \right) \left( \int_{-h}^0 \vert \partial_j \partial_z v^0_k \vert \vert \upet_k \vert^5 \, dz \right) \, d\Mc_0 \\
		&\leq C \int_{\Mc_0} \vert \upet_j \vert_{L^2_z} \vert \partial_j \partial_z v^0_k \vert_{L^2_z} \vert (\upet_k)^5 \vert_{L^2_z} \, d\Mc_0 \leq C \vert \upet \vert_{L^6_x L^2_z} \Vert \partial_z v^0 \Vert \vert (\upet)^5 \vert_{L^3_x L^2_z}\\
		&\leq C \vert \upet \vert_{L^6}^2 \Vert \partial_z v^0 \Vert \left( \vert \grad_3 (\upet)^3 \vert^{4/3} + \vert \upet \vert_{L^6}^4 \right)\\
		&\leq \eta \vert \grad_3 (\upet)^3 \vert^2 + C_\eta \vert \upet \vert_{L^6} \left( \Vert \partial_z v^0 \Vert^{3} + \Vert \partial_z v^0 \Vert \right).
	\end{split}
\end{equation}
Employing the anisotropic estimates \eqref{eq:anis.1} and \eqref{eq:anis.5} we have
\begin{equation}
	\label{eq:l6.i2.122}
	\begin{split}
		I_2^{122} &\leq C \int_{\Mc_0} \left( \int_{-h}^0 \vert \upet_j \vert \, dz' \right) \left( \int_{-h}^0 \vert \partial_z v^0_k \vert \vert (\upet_k)^4 \partial_j \upet_k \vert \, dz \right) \, d\Mc_0\\
		&\leq C \int_{\Mc_0} \vert \upet_j \vert_{L^2_z} \vert \partial_z v^0_k \vert_{L^6_z} \vert (\upet_k)^2 \partial_j \upet_k \vert_{L^2_z} \vert (\upet_k)^2 \vert_{L^2_z} \, d\Mc_0\\
		&\leq C \vert \upet \vert_{L^{12}_x L^2_z} \vert \partial_z v^0 \vert_{L^6} \vert \grad_3 (\upet)^3 \vert \vert (\upet)^2 \vert_{L^4_x L^3_z}\\
		&\leq C \left( \vert \upet \vert_{L^6}^2 \Vert \partial_z v^0 \Vert \vert \grad_3 (\upet)^3 \vert^{4/3} + \vert \upet \vert_{L^6}^{5/2} \Vert \partial_z v^0 \Vert \vert \grad_3 (\upet)^3 \vert^{7/6} + \vert \upet \vert_{L^6}^3 \Vert \partial_z v^0 \Vert \vert \grad_3 (\upet)^3 \vert \right)\\
		&\leq \eta \vert \grad_3 (\upet)^3 \vert^2 + C_{\eta} \vert \upet \vert_{L^6}^6 \left( \Vert \partial_z v^0 \Vert^{3} + \Vert \partial_z v^0 \Vert^2 \right) + C_\eta \vert v \vert_{L^6}^{35/8} \Vert \partial_z v^0 \Vert^{7/4}.
	\end{split}
\end{equation}
Collecting the estimates (\ref{eq:l6.i2.11}-\ref{eq:l6.i2.122}) we obtain
\begin{equation}
	\label{eq:l6.i2.1}
	I_2^1 \leq \eta \vert \grad_3 (\upet)^3 \vert^2 + C_\eta \left( 1 + \vert v \vert_{L^6}^6 \right) \left(1 + \Vert \partial_z v^0 \Vert^3 \right).
\end{equation}

Continuing with $I_2^2$ we recall the estimate \eqref{eq:j2split} and therefore
\[
	I_2^2 \leq 6 \left( \Jc_2^1(\upe, \upe, v^0) + \Jc_2^2(\upe, \upe, v^0) + \Jc_2^3(\upe, \upe, v^0) \right).
\]
By the H\"{o}lder inequality, the Ladyzhenskaya inequality and the anisotropic estimate \eqref{eq:ct63} and the Young inequality we have
\begin{equation}
	\label{eq:l6.i2.21}
	\begin{split}
		\Jc_2^1(\upet, \upet, v^0) &\leq C \int_{\Mc_0} \vert \vb^0 \vert \vert \grad_3 (\upet)^3 \vert_{L^2_z} \vert \upet \vert_{L^6_z}^3 \, d\Mc_0 \leq C \vert \vb^0 \vert_{L^4} \vert \grad_3 (\upet)^3 \vert \vert v \vert_{L^{12}_x L^6_z}^3\\
		&\leq C \vert \vb^0 \vert^{1/2}_{L^2(\Mc_0)} \Vert \vb^0 \Vert^{1/2}_{H^1(\Mc_0)} \vert \grad_3 (\upet)^3 \vert \vert \upet \vert_{L^6}^{3/2} \Vert (\upet)^3 \Vert^{1/2}\\
		&\leq \eta \vert \grad_3 (\upet)^3 \vert^2 + C_\eta \vert \upet \vert_{L^6}^6 \left( \vert U^0 \vert^2 \Vert U^0 \Vert^2 + \vert U^0 \vert \Vert U^0 \Vert \right).
	\end{split}
\end{equation}
The next term can be estimated in exactly the same manner as the above term, therefore we get
\begin{equation}
	\label{eq:l6.i2.22}
	\Jc_2^2(\upe, \upe, v^0) \leq \eta \vert \grad_3 (\upet)^3 \vert^2 + C_\eta \vert \upet \vert_{L^6}^6 \left( \vert U^0 \vert^2 \Vert U^0 \Vert^2 + \vert U^0 \vert \Vert U^0 \Vert \right).
\end{equation}
Regarding the final part of $I_2^2$ we employ the H\"{o}lder inequality, the anisotropic estimates \eqref{eq:ct64} and \eqref{eq:ct65} and the Sobolev embedding $W^{1, 2} \hook L^6$ to get
\begin{equation}
	\label{eq:l6.i2.23}
	\begin{split}
		\Jc_2^3(\upet, \upet, v^0) &\leq \int_{\Mc_0} \vert \upet \vert_{L^2_z} \vert \vt^0 \vert_{L^2_z} \vert \grad_3 (\upet)^3 \vert_{L^2_z} \vert \upet \vert_{L^4_z}^2 \, d\Mc_0\\
		&\leq \vert \upet \vert_{L^8_x L^2_z} \vert \vt^0 \vert_{L^8_x L^2_z} \vert \grad_3 (\upet)^3 \vert \vert \upet \vert_{L^8_x L^4_z}^2\\
		&\leq C \Vert U^0 \Vert \Vert \Re \Vert^{3/4} \vert \upet \vert_{L^6}^{9/4} \vert \grad_3 (\upet)^3 \vert\\
		&\leq \eta \vert \grad_3 (\upet)^3 \vert^2 + C_\eta \vert \upet \vert_{L^6}^{9/2} \Vert U^0 \Vert^2 \Vert \Re \Vert^{3/2}.
	\end{split}
\end{equation}
Collecting the estimates (\ref{eq:l6.i2.21}-\ref{eq:l6.i2.23}) we obtain
\begin{multline}
	\label{eq:l6.i2.2}
	I_2^2 \leq \eta \vert \grad_3 (\upet)^3 \vert^2 + C_\eta \vert \upet \vert_{L^6}^6 \left(1 + \vert U^0 \vert^2 \Vert U^0 \Vert^2 + \Vert \Re \Vert^2 + \Vert U^0 \Vert^8 \right)\\
	+ C_\eta \left(1 + \Vert \Re \Vert^2 + \Vert U^0 \Vert^8 \right).
\end{multline}

The remaining part of the term $I_2$, that is $I_2^3$, can be estimated using \eqref{eq:j2split} as
\[
	I_2^3 \leq 6 \left( \Jc_2^1(\upet, v^0, \upet) + \Jc_2^2(\upet, v^0, \upet) + \Jc_2^3(\upet, v^0, \upet) \right)
\]
By the H\"{o}lder inequality, the Gagliardo-Nirenberg inequality both in 2D and 3D and the aniso\-tropic estimate \eqref{eq:anis.1} for $\grad_3 \vt^0$ we obtain
\begin{equation}
	\label{eq:l6.i2.31}
	\begin{split}
		\Jc_2^1(\upe, v^0, \upe) &\leq \int_{\Mc_0} \vert \vb^\varepsilon \vert \vert \grad_3 \vt^0 \vert_{L^2_z} \vert \upet \vert_{L^{10}_z} \, d\Mc_0 \leq \vert \upeb \vert_{L^3(\Mc_0)} \vert \grad_3 \vt^0 \vert_{L^6_x L^2_z} \vert \upet \vert_{L^{10}}^3\\
		&\leq C \vert \Re \vert^{2/3} \Vert \Re \Vert^{1/3} \Vert U^0 \Vert^{1/3} \Vert U^0 \Vert_{H^2}^{2/3} \vert \upet \vert_{L^6}^2 \Vert (\upet)^3 \Vert\\
		&\leq \eta \vert \grad_3 (\upet)^3 \vert^2 + C_\eta \vert \upet \vert_{L^6}^4 \left( \vert \Re \vert^4 \Vert \Re \Vert^2 + \Vert U^0 \Vert \Vert U^0 \Vert_{H^2}^2 \right)\\
		&\hphantom{\leq \ } + C \vert \upet \vert_{L^6}^5 \left( \vert \Re \vert \Vert \Re \Vert^{1/2} + \Vert U^0 \Vert \Vert U^0 \Vert_{H^2}^2 \right).
	\end{split}
\end{equation}
Next we employ the H\"{o}lder inequality, the Gagliardo-Nierenberg inequality in 2D and the aniso\-tropic estimate \eqref{eq:ct63} and get
\begin{equation}
	\label{eq:l6.i2.32}
	\begin{split}
		\Jc_2^2(\upe, v^0, \upe) &\leq C \int_{\Mc_0} \vert \upeb \vert \vert \grad_3 (\upeb)^3 \vert_{L^2_z} \vert \upeb \vert_{L^6_z} \vert \vt^0 \vert_{L^6_z} \, d\Mc_0\\
		&\leq C \vert \upeb \vert_{L^6(\Mc_0)} \vert \grad_3 (\upet)^3 \vert \vert \upeb \vert_{L^{12}_x L^6_z}^2 \vert \vt^0 \vert_{L^6}\\
		&\leq C \vert \Re \vert^{1/3} \Vert \Re \Vert^{2/3} \vert \grad_3 (\upet)^3 \vert \vert \upet \vert_{L^6} \Vert (\upet)^3 \Vert^{1/3} \Vert U^0 \Vert\\
		&\leq \eta \vert \grad_3 (\upet)^3 \vert^2 + C_\eta \vert \upet \vert_{L^6}^3 \vert \Re \vert \Vert \Re \Vert^2 \Vert U^0 \Vert^3\\
		&\hphantom{\leq \ } + C_\eta \vert \upet \vert_{L^6}^4 \vert \Re \vert^{2/3} \Vert \Re \Vert^{4/3} \Vert U^0 \Vert^2.
	\end{split}
\end{equation}
The final part of $I_2^3$ is estimated in exactly the same way as $\Jc_2^3(\upet, \upet, v^0)$. Therefore we have
\begin{equation}
	\label{eq:l6.i2.33}
	\Jc_2^3(\upe, v^0, \upe) \leq \eta \vert \grad_3 (\upet)^3 \vert^2 + C_\eta \vert \upet \vert_{L^6}^{9/2} \Vert U^0 \Vert^2 \Vert \Re \Vert^{3/2}.
\end{equation}
Collecting the estimates (\ref{eq:l6.i2.31}-\ref{eq:l6.i2.33}) and using the Young inequality we deduce
\begin{multline}
	\label{eq:l6.i2.3}
	I_2^3 \leq \eta \vert \grad_3 (\upet)^3 \vert^2 + C_\eta \left( 1 + \vert \upet \vert_{L^6}^6 \right)\left(1 + \vert \Re \vert^4 \Vert \Re \Vert^2 + \Vert U^0 \Vert \Vert U^0 \Vert_{H^2}^2\right)\\
	 + C_\eta \left( 1 + \vert \upet \vert_{L^6}^6 \right)\left( \vert \Re \vert \Vert \Re \Vert^2 \Vert U^0 \Vert^3 + \Vert \Re \Vert^{2} + \Vert U^0 \Vert^6 \right)
\end{multline}
This finishes the estimates of $I_2$.

Let $K \geq 0$ and let $\tau_a$ and $\tau_b$ be stopping times such that $0 \leq \tau_a \leq \tau_b \leq t \wedge \tau_{K, \varepsilon}^w \wedge \tau^0_K$. Using the Burkholder-Davis-Gundy inequality \eqref{eq:bdg} and the bound \eqref{eq:sigma1.l6} on $\Rc \sigma_1$ in $L_2(\Uc, L^6)$ we get
\begin{equation}
	\label{eq:l6.i3}
	\begin{split}
		\Eb \sup_{s \in [\tau_a, \tau_b]} &\left| \int_{\tau_a}^s \sum_{k=1}^\infty I_3^k \, dW^k_1 \right|\\
		&\leq 6\lambda^{-1}(\varepsilon) C_{BDG} \Eb \left( \int_{\tau_a}^{\tau_b} \sum_{k=1}^\infty \left[ \int_{\Mc} \vert \upet \vert^5 \vert \Rc \sigma_1\left( U^0 + \sqrt{\varepsilon} \lambda(\varepsilon) \Re \right) e_k \vert \, d\Mc \right] \, ds \right)^{1/2}\\
		&\leq C\lambda^{-1}(\varepsilon) \Eb \left( \int_{\tau_a}^{\tau_b} \vert \upet \vert_{L^6}^{10} \sum_{k=1}^\infty \vert \Rc \sigma_1\left( U^0 + \sqrt{\varepsilon} \lambda(\varepsilon) \Re \right) e_k \vert_{L^6}^2 \, ds \right)^{1/2}\\
		&\leq C \lambda^{-1}(\varepsilon) \Eb \left( \int_{\tau_a}^{\tau_b} \vert \upet \vert_{L^6}^{10} \left( 1 + \vert \vt^0 \vert^2_{L^6} + \varepsilon \lambda^{-2}(\varepsilon) \vert \upet \vert_{L^6}^2 \right) \, ds \right)^{1/2}\\
		&\leq \eta \Eb \sup_{s \in [\tau_a, \tau_b]} \vert \upet \vert_{L^6}^6 + C_\eta \lambda^{-2}(\varepsilon) \Eb \int_{\tau_a}^{\tau_b} \left(1 + \vert \upet \vert_{L^6}^6 \right) \left( 1 + \Vert U^0 \Vert^2 \right) \, ds.
	\end{split}
\end{equation}

Regarding the It\^{o} correction term, from the bound \eqref{eq:sigma1.l6} on $\Rc \sigma_1$ in $L_2(\Uc, L^6)$ we deduce
\begin{equation}
	\label{eq:l6.i4}
	\begin{split}
		\sum_{k=1}^\infty I_4^k &\leq C \lambda^{-2}(\varepsilon) \vert \upet \vert_{L^6}^4 \sum_{k=1}^\infty \vert \Rc \sigma_1\left( U^0 + \sqrt{\varepsilon} \lambda(\varepsilon) \Re \right) e_k \vert_{L^6}^2\\
		&\leq C \lambda^{-2}(\varepsilon) \left(1 + \vert \upet \vert_{L^6}^6 \right) \left( 1 + \Vert U^0 \Vert^2 \right).
	\end{split}
\end{equation}

Finally, choosing $\eta > 0$ sufficiently small we use the estimates \eqref{eq:l6.i1}, \eqref{eq:l6.i2.1}, \eqref{eq:l6.i2.2}, \eqref{eq:l6.i2.3}, \eqref{eq:l6.i3} and \eqref{eq:l6.i4} to infer
\begin{align*}
	\Eb \left[ \sup_{s \in [\tau_a, \tau_b]} \vert \upet \vert^6 + \int_{\tau_a}^{\tau_b} \vert \grad_3 (\upet)^3 \vert^2 \, ds \right] \leq C \Eb \int_{\tau_a}^{\tau_b} \left(1 + \vert \upet \vert_{L^6}^6 \right) \Phi(s) \, ds,
\end{align*}
where
\begin{equation*}
	\Phi(s) = 1 + \vert \Re \vert^4 \Vert \Re \Vert^2 + \vert \Re \vert^4 + \Vert \Re \Vert^2 + \Vert \partial_z U^0 \Vert^3 + \Vert U^0 \Vert^8 + \vert \Re \vert \Vert \Re \Vert^2 \Vert U^0 \Vert^3.
\end{equation*}
We observe that the definitions of the stopping times $\tau_K^0$ and $\tau_K^{w, \varepsilon}$ and the "uniform" convergence \eqref{eq:mpd.tau.0.convergence} justify the use of the uniform stochastic Gronwall lemma from Proposition \ref{prop:stoch.Gronwall}, which finishes the proof. 
\end{proof}

\begin{proposition}
\label{prop:mdp.grad.vb}
Let $K \geq 0$ and let $\tau_K^{\grad, \varepsilon}$ be the stopping time defined by
\begin{equation*}
	\tau_K^{\grad, \varepsilon} = \inf \left\lbrace s \geq 0 \mid \int_0^s \Vert \upeb \Vert_{H^1(\Mc_0)}^4 \, dr \geq K \right\rbrace.
\end{equation*}
Then $\tau_K^{\grad, \varepsilon} \to \infty$ $\Pb$-almost surely as $K \to \infty$ for all $\varepsilon \in (0, \varepsilon_0]$. Moreover for all $\tilde{t} \geq 0$ one has
\[
	\lim_{K \to \infty} \Pb \left( \left\lbrace \tau_K^{\grad, \varepsilon} \leq \tilde{t} \right\rbrace \right) = 0 \ \text{uniformly w.r.t.\ } \varepsilon \in (0, \varepsilon_0].
\]
\end{proposition}

\begin{proof}
Applying the It\^{o} lemma \cite[Theorem A.1]{brzezniakslavik} to \eqref{eq:upeb} with the function $\vert A_S^{1/2} P_{\Hbar} \cdot \vert^4$ and using the cancellation property
\[
	\int_{\Mc_0} \left( u \cdot \grad \right) v \cdot \laplace v	\, d\Mc_0 = 0
\]
holding for $u \in H^2(\Mc_0)$ with $\div u = \div v = 0$ and periodic boundary conditions, we obtain
\begin{align*}
	d&\vert \grad \upeb \vert^4 + 4\mu \vert \grad \upeb \vert^2 \vert \laplace \upeb \vert^2 \, dt\\
	&\leq 4 \vert \grad \upeb\vert^2 \bigg[ \beta_T g \left| \left( \laplace \upeb, P_{\Hbar} \Ac_2 \int_{z}^0 \grad \Upe \right) \right| + \sqrt{\varepsilon} \lambda(\varepsilon) \left| \left( P_{\Hbar} \Jc_1(\upe, \upe), \laplace \upeb \right) \right|\\
	&\hphantom{\leq \vert \grad \upeb\vert^2 \bigg[ \ } + \left| \left( P_{\Hbar} \Jc_1(v^0, \upe), \laplace \upeb \right) \right| + \left| \left(P_{\Hbar} B_1(\upeb, v^0), \laplace \upeb \right) \right| + \left| \left( P_{\Hbar} \Jc_1(\upe, v^0), \laplace \upeb \right) \right| \bigg] \, dt\\
	&\hphantom{\leq \ } + 6 \lambda^{-2}(\varepsilon) \vert \grad \upeb\vert^2 \Vert \grad \Ac_2 \sigma_1\left( U^0 + \sqrt{\varepsilon} \lambda(\varepsilon) R^\varepsilon \right) \Vert_{L_2(\Uc, L^2(\Mc_0))}^2 \, dt\\
	&\hphantom{\leq \ } + 4 \lambda^{-1}(\varepsilon) \vert \grad \upeb\vert^2 \left( \grad \upeb, \grad \Ac_2 \sigma_1\left( U^0 + \sqrt{\varepsilon} \lambda(\varepsilon) R^\varepsilon \right) \, dW_1 \right)\\
	&= \sum_{k=1}^6 I_k \, dt + I_7 \, dW_1.
\end{align*}
The term $I_1$ can be estimated by the H\"{o}lder inequality. We get
\begin{equation}
	\label{eq:mdp.vb.i1}
	I_1 \leq \eta \vert \grad \upeb \vert^2 \vert \laplace \upeb \vert^2 + C_\eta \vert \grad \upeb \vert^2 \Vert \Re \Vert^2
\end{equation}
for some $\eta > 0$ fixed precisely determined later. Next we argue as in \cite[Section 3.3]{cao2007} and use the H\"{o}lder inequality to get
\begin{equation}
	\label{eq:mdp.vb.i2}
	\begin{split}
		I_2 &\leq C \sqrt{\varepsilon} \lambda(\varepsilon) \vert \grad \upeb \vert^2 \vert \laplace \upeb \vert \vert \grad_3 (\upet)^3 \vert^{1/2} \Vert \Re \Vert^{1/2}\\
		&\leq \eta \vert \grad \upeb \vert^2 \vert \laplace \upeb \vert^2 + C_\eta \varepsilon \lambda^2(\varepsilon) \left( \vert \grad \upeb \vert^4 \vert \grad_3 (\upet)^3 \vert^2 + \Vert \Re \Vert^2 \right).
	\end{split}
\end{equation}
We deal with $I_3$ using the H\"{o}lder inequality and Ladyzhenskaya inequality. We get
\begin{equation}
	\label{eq:mdp.vb.i3}
	\begin{split}
		I_4 &\leq C \vert \grad \upeb \vert^2 \vert \laplace \upeb \vert \vert (\upeb \cdot \grad) v^0 \vert \leq C \vert \grad \upeb \vert^2 \vert \laplace \upeb \vert \vert \upeb \vert_{L^4} \vert \grad \vb^0 \vert_{L^4}\\
		&\leq C \vert \grad \upeb \vert^2 \vert \laplace \upeb \vert \vert \upeb \vert^{1/2} \Vert \upeb \Vert^{1/2} \vert \grad \vb^0 \vert^{1/2} \Vert \grad \vb^0 \Vert^{1/2}\\
		&\leq \eta \vert \grad \upeb \vert^2 \vert \laplace \upeb \vert^2 + C_\eta \vert \grad \upeb \vert^2 \left( \vert \Re \vert^2 \Vert \Re \Vert^2 + \Vert U^0 \Vert^2 \Vert U^0 \Vert^2_{H^2} \right).
	\end{split}
\end{equation}
The term $I_4$ can be estimated in exactly the same wat as $I_3$, therefore
\begin{equation}
	\label{eq:mdp.vb.i4}
	I_4 \leq \eta \vert \grad \upeb \vert^2 \vert \laplace \upeb \vert^2 + C_\eta \vert \grad \upeb \vert^2 \left( \vert \Re \vert^2 \Vert \Re \Vert^2 + \Vert U^0 \Vert^2 \Vert U^0 \Vert^2_{H^2} \right)
\end{equation}
Regarding $I_5$ we split the term by the H\"{o}lder inequality as follows:
\begin{equation*}
	I_5 \leq C \vert \grad \upeb \vert^2 \vert \laplace \upeb \vert \left( \vert \Ac_2 \left[ (\upet \cdot \grad) \vt^0 \right] \vert + \vert \Ac_2 \left[ (\div \upet) \vt^0 \right] \vert \right) = I_5^1 + I_5^2.
\end{equation*}
By the H\"{o}lder inequality, the Young inequality and the anisotropic estimate \eqref{eq:anis.1} applied to $\grad \vt^0 \in H^1(\Mc; \Rb^2)$ we get
\begin{align}
	\nonumber
	I_5^1 &\leq C \vert \grad \upeb \vert^2 \vert \laplace \upeb \vert \left( \int_{\Mc_0} \vert \upet \vert_{L^2_z}^2 \vert \grad \vt^0 \vert_{L^2_z}^2 \, d\Mc_0 \right)^{1/2} \leq C \vert \grad \upeb \vert^2 \vert \laplace \upeb \vert \vert \upet \vert_{L^6_x L^2_z} \vert \grad \vt^0 \vert_{L^3_x L^2_z}\\
	\nonumber
	&\leq C \vert \grad \upeb \vert^2 \vert \laplace \upeb \vert \vert \upet \vert_{L^6} \vert \grad \vt^0 \vert^{2/3} \Vert \grad \vt^0 \Vert^{1/3}\\
	\label{eq:mdp.vb.i5.1}
	&\leq \eta \vert \grad \upeb \vert^2 \vert \laplace \upeb \vert^2 + C_\eta \vert \grad \upeb \vert^2 \left( \vert \upet \vert_{L^6}^4 + \Vert \vt^0 \Vert^{8/3} \Vert \vt^0 \Vert_{H^2}^{4/3} \right).
\end{align}
Regarding the term $I_5^2$ we use the boundedness of $\Rc$ in $L^\infty$ to deduce
\begin{equation}
	\label{eq:mdp.vb.i5.2}
	I_5^2 \leq C \vert \grad \upeb \vert^2 \vert \laplace \upeb \vert \vert \vt^0 \vert_{L^\infty} \vert \grad \upet \vert \leq \eta \vert \grad \upeb \vert^2 \vert \laplace \upeb \vert^2 + C_\eta \vert \grad \upeb \vert^2 \Vert \Re \Vert^2 \vert v^0 \vert_{L^\infty}^2.
\end{equation}
We estimate the It\^{o} correction term using the estimate \eqref{eq:sigma.vb} and the Young inequality by
\begin{equation}
	\label{eq:mdp.vb.i6}
	I_6 \leq C \lambda^{-2}(\varepsilon) \vert \grad \upeb \vert^2 \left(1 + \Vert U^0 \Vert^2 + \vert \laplace \vb^0 \vert^2 \right) + C \varepsilon \vert \grad \upeb \vert^2 \Vert \Re \Vert^2 + C \varepsilon \eta_2 \vert \grad \upeb \vert^2 \vert \laplace \upeb \vert^2.
\end{equation}
Let $K \geq 0$ and let $\tau_a$ and $\tau_b$ be stopping times such that $0 \leq \tau_a \leq \tau_b \leq t \wedge \tau^{6, \varepsilon}_K \wedge \tau^{w, \varepsilon}_K \wedge \tau^0_K$. Using the Burkholder-Davis-Gundy inequality and the bound \eqref{eq:sigma.vb} we may estimate the stochastic term similarly as in \eqref{eq:w.i1} and \eqref{eq:w.i2} by
\begin{equation}
	\label{eq:mdp.vb.i7}
	\begin{split}
		4 \lambda^{-1}(\varepsilon) &\Eb \sup_{s \in [\tau_a, \tau_b]} \left| \int_{\tau_a}^s \vert \grad \upeb \vert^2 \left( \grad \upeb, \grad \Ac_2 \sigma_1\left( U^0 + \sqrt{\varepsilon} \lambda(\varepsilon) R^\varepsilon \right) \, dW_1 \right) \right|\\
		&\leq \left(1 - \delta + \eta\right) \Eb \sup_{s \in [\tau_a, \tau_b]} + C_\eta \lambda^{-2}(\varepsilon) \Eb \int_{\tau_a}^{\tau_b} \vert \grad \upeb \vert^2 \left( 1 + \Vert U^0 \Vert^2 + \Vert U^0 \Vert_{H^2}^2 \right) \, ds\\
		&\hphantom{\leq \ } + C_\eta \varepsilon \Eb \int_{\tau_a}^{\tau_b} \vert \grad \upeb \vert^4 \, ds + \frac{C_\eta \eta_2 \varepsilon}{1 - \delta} \Eb \int_{\tau_a}^{\tau_b} \vert \grad \upeb \vert^2 \vert \laplace \upeb \vert^2 \, ds
	\end{split}
\end{equation}
for some $\delta \in [0, 1)$.

Collecting the estimates (\ref{eq:mdp.vb.i1}-\ref{eq:mdp.vb.i7}), choosing $\eta$ and $\delta$ appropriately small and assuming that $\varepsilon_0$ is sufficiently small we infer
\begin{equation*}
	\Eb \left[ \sup_{s \in [\tau_a, \tau_b]} \vert \grad \upeb \vert^4 + \int_{\tau_a}^{\tau_b} \vert \grad \upeb \vert^2 \vert \laplace \upeb \vert^2 \, ds \right] \leq C \Eb \int_{\tau_a}^{\tau_b} \left(1 + \vert \grad \upeb \vert^4 \right) \Phi(s) \, ds,
\end{equation*}
where
\begin{multline*}
	\Phi(s) = 1 + \Vert \Re \Vert^2 + \vert \grad_3 (\upet)^3 \vert^2 + \vert \Re \vert^2 \Vert \Re \Vert^2 + \Vert U^0 \Vert^4 \Vert U^0 \Vert^2_{H^2} + \vert \upet \vert_{L^6}^6\\
	+ \Vert \Re \Vert^2 \vert v^0 \vert_{L^\infty}^2 + \Vert U^0 \Vert^2 + \Vert U^0 \Vert_{H^2}^2.
\end{multline*}
The rest follow from the uniform stochastic Gronwall lemma, see Proposition \ref{prop:stoch.Gronwall}.
\end{proof}

\begin{proposition}
\label{prop:mdp.vz}
Let $K \geq 0$ and let $\tau_K^{z, \varepsilon}$ be the stopping time defined by
\begin{equation}
	\tau_K^{z, \varepsilon} = \inf \left\lbrace s \geq 0 \mid \int_0^s \vert \partial_z \Re \vert^2 \Vert \partial_z \Re \Vert^2 \, dr \geq K \right\rbrace.
\end{equation}
Then $\tau_K^{z, \varepsilon} \to \infty$ $\Pb$-almost surely as $K \to \infty$ for all $\varepsilon \in (0, \varepsilon_0]$. Moreover for all $\tilde{t} \geq 0$ one has
\[
	\lim_{K \to \infty} \Pb \left( \left\lbrace \tau_K^{z, \varepsilon} \leq \tilde{t} \right\rbrace \right) = 0 \ \text{uniformly w.r.t.\ } \varepsilon \in (0, \varepsilon_0].
\]
\end{proposition}

\begin{proof}
The proof uses elements from the proof \cite[Proposition 5.2]{debussche2012}. Let $p \geq 2$. By the It\^{o} lemma \cite[Theorem A.1]{brzezniakslavik}, integration by parts and the cancellation property \eqref{eq:b.cancellation} we get
\begin{align*}
	d\vert \partial_z \upe \vert^p &+ p(\mu \wedge \nu) \vert \partial_z \upe \vert^{p-2} \vert \grad_3 \partial_z \upe \vert^2 \, dt\\
	&\leq p \vert \partial_z \upe \vert^{p-2} \bigg[ \left| \left( \grad \Upe, \partial_z \upe \right) \right| +  \sqrt{\varepsilon} \lambda(\varepsilon) \left| \left( \partial_z B_2(\upe, \upe), \partial_z \upe \right)\right| + \left| \left( \partial_z B_2(v^0, \upe), \partial_z \upe \right)\right|\\
	&\hphantom{p \vert \partial_z \upe \vert^{p-2} \bigg[ \ } + \left| \left( \partial_z B_2(\upe, v^0), \partial_z \upe \right)\right| + \frac{p-2}{2} \Vert \partial_z \sigma_1\left(U^0 + \sqrt{\varepsilon} \lambda(\varepsilon) \Re \right) \Vert_{L_2(\Uc, L^2)}^2 \bigg] \, dt\\
	&\hphantom{\leq \ } + p \vert \partial_z \upe \vert^{p-2} \left( \partial_z \sigma_1\left(U^0 + \sqrt{\varepsilon} \lambda(\varepsilon) \Re \right) \, dW_1, \partial_z \upe \right)\\
	&= \sum_{k=1}^5 I_k + I_6 \, dW_1.
\end{align*}

From the Cauchy-Schwartz inequality we immediately obtain
\begin{equation}
	\label{eq:eq:mdp.vz.i1}
	I_1 \leq C \vert \partial_z \upe \vert^{p-1} \Vert \Re \Vert.
\end{equation}
Before we proceed to $I_2$, let $u, v \in D(A)$. By the cancellation property \eqref{eq:b.cancellation} and integration by parts we have
\begin{align}
	\nonumber
	\left( \partial_z B_2(u, v), \partial_z v \right) &= \left( \left( \partial_z u \cdot \grad \right) v - \left( \div u \right) \partial_z v, \partial_z v \right) = \int_{\Mc} \partial_z u_j \partial_j v_k \partial_z v_k - \partial_j u_j \partial_z v_k \partial_z v_k \, d\Mc\\
	\label{eq:mdp.vz.split}
	&= - \int_{\Mc} \partial_z \partial_j u_j v_k \partial_z v_k + \partial_z u_j v_k \partial_j \partial_z v_k - 2 u_j \partial_j \partial_z v_k \partial_z v_k \, d\Mc.
\end{align}
Using the H\"{o}lder inequality, \eqref{eq:mdp.vz.split}, the Gagliardo-Nirenberg inequality and the Young inequality we deduce
\begin{equation}
	\label{eq:mdp.vz.i2}
	\begin{split}
		I_2 &\leq C \sqrt{\varepsilon} \lambda(\varepsilon) \vert \partial_z \upe \vert^{p-2} \vert \grad \partial_z \upe \vert \vert \upe \vert_{L^6} \vert \partial_z \upe \vert_{L^3}\\
		&\leq C \sqrt{\varepsilon} \lambda(\varepsilon) \left( \vert \partial_z \grad \upe \vert^{3/2} \vert \partial_z \upe \vert^{p-3/2} \vert \upe \vert_{L^6} + \vert \partial_z \grad \upe \vert^{1/2} \vert \partial_z \upe \vert^{p-1} \vert \upe \vert_{L^6} \right)\\
		&\leq \eta \vert \partial_z \upe \vert^{p-2} \vert \partial_z \grad \upe \vert^2 + C_\eta \vert \partial_z \upe \vert^p \left( \vert \upe \vert_{L^6}^4 + 1 \right)
	\end{split}
\end{equation}
for some $\eta > 0$ small precisely determined later. Using the H\"{o}lder inequality and \eqref{eq:mdp.vz.split} again we obtain
\begin{align*}
	I_3 &\leq C \vert \partial_z \upe \vert^{p-2} \left( \Vert v^0 \Vert_{H^2} \vert \upe \vert_{L^6} \vert \partial_z \upe \vert_{L^3} + \vert \partial_z v^0 \vert_{L^3} \vert \upe \vert_{L^6} \vert \partial_z \grad \upe \vert + \vert v^0 \vert_{L^\infty} \vert \partial_z \grad \upe \vert \vert \partial_z \upe \vert \right)\\
	&= I_3^1 + I_3^2 + I_3^3.
\end{align*}
By the Gagliardo-Nierenberg inequality and the Young inequality we get
\begin{equation}
	\label{eq:mdp.vz.i3.1}
	\begin{split}
		I_3^1 &\leq C \vert \partial_z \upe \vert^{p-2} \Vert v^0 \Vert_{H^2} \vert \upe \vert_{L^6} \left( \vert \partial_z \upe \vert^{1/2} \vert \grad \partial_z \upe \vert^{1/2} + \vert \partial_z \upe \vert \right)\\
		&\leq \eta \vert \partial_z \upe \vert^{p-2} \vert \partial_z \grad \upe \vert^2 + C_\eta \vert \partial_z \upe \vert^{p-4/3} \Vert v^0 \Vert_{H^2}^{4/3} \vert \upe \vert_{L^6}^{4/3} + C \vert \partial_z \upe \vert^{p-1} \Vert v^0 \Vert_{H^2} \vert \upe \vert_{L^6}.
	\end{split}
\end{equation}
Using the Young inequality and the Gagliardo-Nirenberg inequality for the term $I_3^2$ we obtain
\begin{equation}
	\label{eq:mdp.vz.i3.2and3}
	I_3^2 + I_3^3 \leq \eta \vert \partial_z \upe \vert^{p-2} \vert \partial_z \grad \upe \vert^2 + C_\eta \vert \partial_z \upe \vert^{p-2} \vert \upe \vert_{L^6}^2 \Vert v^0 \Vert \Vert v^0 \Vert_{H^2} + C_\eta \vert \partial_z \upe \vert^p \vert v^0 \vert_{L^\infty}^2.
\end{equation}
Proceeding to $I_4$ we use integration by parts to deduce
\begin{align*}
	I_4 &\leq C \vert \partial_z \upe \vert^{p-2} \left[ \left| \left( B_2(\upe, v^0), \partial_{zz} \upe \right) \right| + \left| \left( (\div \upe) \partial_z v^0, \partial_z \upe \right) \right| + \left| \left( w(\upe) \partial_{zz} v^0, \partial_z \upe \right) \right| \right]\\
	&= I_4^1 + I_4^2 + I_4^3.
\end{align*}
Similarly as above we use the H\"{o}lder inequality, the Gagliardo-Nirenberg inequality and the Young inequality to get
\begin{equation}
	\label{eq:mdp.vz.i4.1}
	I_4^1 \leq \eta \vert \partial_z \upe \vert^{p-2} \vert \partial_z \grad \upe \vert^2 + C_\eta \vert \partial_z \upe \vert^{p-2} \vert \upe \vert_{L^6} \Vert v^0 \Vert \Vert v^0 \Vert_{H^2}.
\end{equation}
The term $I_4^2$ can be estimated in exactly the same way as $I_3^1$, therefore we have
\begin{equation}
	\label{eq:mdp.vz.i4.2}
	I_4^2 \leq \eta \vert \partial_z \upe \vert^{p-2} \vert \partial_z \grad \upe \vert^2 + C_\eta \vert \partial_z \upe \vert^{p-4/3} \Vert v^0 \vert_{H^2}^{4/3} \vert \upe \vert_{L^6}^{4/3} + C \vert \partial_z \upe \vert^{p-1} \Vert v^0 \Vert_{H^2} \vert \upe \vert_{L^6}.
\end{equation}
Next we employ integration by parts to get
\begin{multline*}
	- \int_{\Mc} \left( \int_{-h}^0 \div \upe \, dz' \right) \partial_{zz} v^0 \partial_z \upe \, d\Mc_0\\
	= \int_{\Mc_0} \left( \int_{-h}^z \upe_j \, dz' \right) \partial_j \partial_{zz} v^0_k \partial_z \upe_k + \left( \int_{-h}^z \upe_j \, dz' \right) \partial_{zz} v^0_k \partial_j \partial_z \upe_k \, d\Mc.
\end{multline*}
Therefore we have
\begin{align*}
	I_4^3 &\leq C \vert \partial_z \upet \vert^{p-2} \left| \int_{\Mc_0} \left( \int_{-h}^z \upe_j \, dz' \right) \partial_j \partial_{zz} v^0_k \partial_z \upe_k \, d\Mc \right|\\
	&\hphantom{\leq \ } + C \vert \partial_z \upet \vert^{p-2} \left| \int_{\Mc_0} \left( \int_{-h}^z \upe_j \, dz' \right) \partial_{zz} v^0_k \partial_j \partial_z \upe_k \, d\Mc \right|\\
	&= I_4^{31} + I_4^{32}.
\end{align*}
using the H\"{o}lder inequality, the embedding $L^6 \hook L^6_x L^2_z$, the anisotropic estimate \eqref{eq:anis.1} and the Young inequality we get
\begin{equation}
	\label{eq:mdp.vz.i4.31}
	\begin{split}
		I_4^{31} &\leq C \vert \partial_z \upe \vert^{p-2} \int_{\Mc} \vert \upe \vert_{L^2_z} \vert \partial_{zz} v^0 \vert_{L^2_z} \vert \grad \partial_z \upe \vert_{L^2_z} \, d\Mc_0\\
		&\leq C \vert \partial_z \upe \vert^{p-2} \vert \upe \vert_{L^6_x L^2_z} \vert \partial_{zz} v^0 \vert_{L^3_x L^2_z} \vert \partial_z \grad \upe \vert\\
		&\leq C \vert \partial_z \upe \vert^{p-2} \vert \partial_z \grad \upe \vert \vert \upe \vert_{L^6} \Vert \partial_z v^0 \Vert^{2/3} \Vert \partial_z v^0 \Vert_{H^2}^{1/3}\\
		&\leq \eta \vert \partial_z \upe \vert^{p-2} \vert \partial_z \grad \upe \vert^2 + C_\eta \vert \partial_z \upe \vert^{p-2} \vert \upe \vert_{L^6}^2 \Vert \partial_z v^0 \Vert^{4/3} \Vert \partial_z v^0 \Vert_{H^2}^{2/3}.
	\end{split}
\end{equation}
Similarly we deduce
\begin{equation}
	\label{eq:mdp.vz.i4.32}
	\begin{split}
		I_4^{32} &\leq C \vert \partial_z \upe \vert^{p-2} \int_{\Mc} \vert \upe \vert_{L^2_z} \vert \grad \partial_{zz} v^0 \vert_{L^2_z} \vert \partial_z \upe \vert_{L^2} \, d\Mc\\
		&\leq C \vert \partial_z \upe \vert^{p-2} \vert \upe \vert_{L^6_x L^2_z} \Vert \partial_z v^0 \Vert_{H^2} \vert \partial_z \upe \vert_{L^3_x L^2_z}\\
		&\leq C \vert \partial_z \upe \vert^{p-2} \vert \upe \vert_{L^6} \Vert \partial_z v^0 \Vert_{H^2} \left( \vert \partial_z \upe \vert^{2/3} \vert \grad_3 \partial_z \upe \vert^{1/3} + \vert \partial_z \upe \vert \right)\\
		&\leq \eta \vert \partial_z \upe \vert^{p-2} \vert \partial_z \grad \upe \vert^2 + C_\eta \vert \partial_z \upe \vert^{p-6/5} \Vert \partial_z v^0 \Vert_{H^2}^{6/5} \vert \upe \vert_{L^6}^{6/5}.
	\end{split}
\end{equation}
The It\^{o} correction term is estimated using the bound \eqref{eq:sigma.vz} and the Young inequlity by
\begin{multline}
	\label{eq:mdp.vz.i5}
	I_5 \leq C \lambda^{-2}(\varepsilon) \vert \partial_z \upe \vert^{p-2} \left( 1 + \Vert U^0 \Vert^2 + \Vert U^0 \Vert_{H^2}^2 \right)\\
	+ C \varepsilon \vert \partial_z \upe \vert^p + p(p-1) \varepsilon \eta_0 \vert \partial_z \upe \vert^{p-2} \vert \grad_3 \partial_z \upe \vert^2.
\end{multline}
Let $K \geq 0$ and let $\tau_a$ and $\tau_b$ be stopping times such that
\[
	0 \leq \tau_a \leq \tau_b \leq t \wedge \tau^{\grad, \varepsilon}_K \wedge \tau^{6, \varepsilon}_K \wedge \tau^{w, \varepsilon}_K \wedge \tau^0_K.
\]
Proceeding similarly as in \eqref{eq:w.i1} and \eqref{eq:w.i2} we employ the Burkholder-Davis-Gundy inequality and the bound \eqref{eq:sigma.vz} on $\partial_z \sigma_1$ in $L_2(\Uc, L^2)$ and deduce
\begin{equation}
	\label{eq:mdp.vz.i6}
	\begin{split}
		p \lambda^{-1}(\varepsilon) &\Eb \sup_{s \in [\tau_a, \tau_b]} \left| \int_{\tau_a}^s \vert \partial_z \upe \vert^{p-2} \left( \partial_z \upe, \partial_z \sigma_1\left( U^0 + \sqrt{\varepsilon} \lambda(\varepsilon) R^\varepsilon \right) \, dW_1 \right) \right|\\
		&\leq \left(1 - \delta + \eta\right) \Eb \sup_{s \in [\tau_a, \tau_b]} \vert \partial_z \upe \vert^p + C_\eta \lambda^{-2}(\varepsilon) \Eb \int_{\tau_a}^{\tau_b} \vert \partial_z \upe \vert^{p-2} \left( 1 + \Vert U^0 \Vert_{H^2}^2 \right) \, ds\\
		&\hphantom{\leq \ } + C_\eta \varepsilon \Eb \int_{\tau_a}^{\tau_b} \vert \partial_z \upe \vert^{p} \, ds + \frac{C_\eta \eta_3 \varepsilon}{1 - \delta} \Eb \int_{\tau_a}^{\tau_b} \vert \partial_z \upe \vert^{p-2} \vert \grad_3 \partial_z \upe \vert^2 \, ds
	\end{split}
\end{equation}
for some $\delta \in (\eta, 1)$.

Collecting the above estimates, choosing $\delta$ and $\eta$ sufficiently small ans assuming that $\eta_3$ or $\varepsilon_0$ are sufficiently small we get
\begin{equation*}
	\Eb \left[ \sup_{s \in [\tau_a, \tau_b]} \vert \partial_z \upe \vert^p + \int_{\tau_a}^{\tau_b} \vert \partial_z \upe \vert^{p-2} \vert \grad_3 \partial_z \upe \vert^2 \, ds \right] \leq C \Eb \int_{\tau_a}^{\tau_b} \left(1 + \vert \partial_z \upe \vert^p \right) \Phi(s) \, ds,
\end{equation*}
where
\begin{equation*}
	\Phi(s) = 1 + \Vert \Re \Vert^2 + \vert \upe \vert_{L^6}^4 + \Vert U^0 \Vert_{H^2}^2 + \Vert U^0 \Vert^2 \Vert U^0 \Vert_{H^2}^2 + \Vert \partial_z U^0 \Vert^4 \Vert \partial_z U^0 \Vert_{H^2}^2 + \Vert \partial_z U^0 \Vert_{H^2}^2.
\end{equation*}
The claim follows from the uniform stochastic Gronwall lemma from Proposition \ref{prop:stoch.Gronwall}.
\end{proof}

The proof of the following proposition is similar to the proof of Proposition \ref{prop:mdp.vz} and \cite[Proposition 4.6]{brzezniakslavik} and is therefore omitted.

\begin{proposition}
\label{prop:mdp.T}
Let $K \geq 0$ and let $\tau_K^{T, \varepsilon}$ be the stopping time defined by
\begin{equation}
	\tau_K^{T, \varepsilon} = \inf \left\lbrace s \geq 0 \mid \sup_{r \in [0, s]} \vert \partial_z \Upe \vert^4 + \int_0^s \vert \partial_z \Upe \vert^2 \Vert \partial_z \Upe \Vert^2 \, dr \geq K \right\rbrace.
\end{equation}
Then $\tau_K^{T, \varepsilon} \to \infty$ $\Pb$-almost surely as $K \to \infty$ for all $\varepsilon \in (0, \varepsilon_0]$. Moreover for all $\tilde{t} \geq 0$ one has
\[
	\lim_{K \to \infty} \Pb \left( \left\lbrace \tau_K^{T, \varepsilon} \leq \tilde{t} \right\rbrace \right) = 0 \ \text{uniformly w.r.t.\ } \varepsilon \in (0, \varepsilon_0].
\]
\end{proposition}

\begin{proposition}
\label{prop:mdp.R}
Let $p \geq 2$, $K \geq 0$ and let $\tau_K^{U, \varepsilon, p}$ be the stopping time defined by
\begin{equation}
	\tau_K^{R, \varepsilon, p} = \inf \left\lbrace s \geq 0 \mid \sup_{r \in [0, s]} \Vert \Re \Vert^p + \int_0^s \Vert \Re \Vert^{p-2} \Vert \Re \Vert^2_{H^2} \, dr \geq K \right\rbrace.
\end{equation}
Then $\tau_K^{R, \varepsilon, p} \to \infty$ $\Pb$-almost surely as $K \to \infty$ for all $\varepsilon \in (0, \varepsilon_0]$. Moreover for all $\tilde{t} \geq 0$ one has
\[
	\lim_{K \to \infty} \Pb \left( \left\lbrace \tau_K^{R, \varepsilon, p} \leq \tilde{t} \right\rbrace \right) = 0 \ \text{uniformly w.r.t.\ } \varepsilon \in (0, \varepsilon_0].
\]
\end{proposition}

\begin{proof}
By the It\^{o} lemma \cite[Theorem A.1]{brzezniakslavik} and integrating by parts we have
\begin{align*}
	d&\vert A^{1/2} \Re \vert^p + p\left( \mu \wedge \nu \right) \vert A^{1/2} \Re \vert^{p-2} \vert A \Re \vert^2 \, dt\\
	&\leq p \vert A^{1/2} \Re \vert^{p-2} \bigg[ \sqrt{\varepsilon} \lambda(\varepsilon) \left| \left( B(\Re, \Re), A\Re \right) \right| + \left| \left( B(U^0, \Re), A\Re \right) \right| + \left| \left( B(\Re, U^0), A\Re \right) \right|\\
	&\hphantom{\leq 2 \bigg[ \ } + \left| \left( F_d(\Re), A\Re \right)  \right|  +  \tfrac{p-1}2 \lambda^{-2}(\varepsilon) \Vert \sigma\left( U^0 + \sqrt{\varepsilon} \lambda(\varepsilon) \Re \right) \Vert_{L_2(\Uc, V)}^2 \bigg] \, dt\\
		&\hphantom{\leq \ } + p \lambda^{-1}(\varepsilon) \vert A^{1/2} \Re \vert^{p-2} \left( A^{1/2} \sigma\left( U^0 + \sqrt{\varepsilon} \lambda(\varepsilon) \Re \right) \, dW, A^{1/2} \Re \right)\\
		&= \sum_{k=1}^5 I_k \, dt + I_6 \, dW.
\end{align*}
By the estimate \eqref{eq:b.estimatel6} and the Young inequality we have
\begin{equation}
	\label{eq:mdp.R.i1}
	I_1 \leq \eta \Vert \Re \Vert^{p-2} \Vert \Re \Vert_{H^2}^2 + C_\eta \varepsilon \lambda^{2}(\varepsilon) \Vert \Re \Vert^p \left( \vert \upe \vert_{L^6}^4 + \vert \partial_z \Re \vert^2 \Vert \partial_z \Re \Vert^2 \right),
\end{equation}
where $\eta > 0$ will be determined later. From the estimate \eqref{eq:b.estimate2} and the Young inequality we get
\begin{equation}
	\label{eq:mdp.R.i234}
	I_2 + I_3 + I_4 \leq \eta \Vert \Re \Vert^{p-2} \Vert \Re \Vert_{H^2}^2 + C_\eta \Vert \Re \Vert^p \left( 1 + \Vert U^0 \Vert^2 \Vert U^0 \Vert_{H^2}^2 \right).
\end{equation}
For the It\^{o} correction term we have
\begin{equation}
	\label{eq:mdp.R:i5}
	I_5 \leq C \lambda^{-2}(\varepsilon) \Vert \Re \Vert^{p-2} \left(1 + \Vert U^0 \Vert_{H^2}^2 \right) + C \varepsilon \Vert \Re \Vert^p + C \varepsilon \eta_1 \Vert \Re \Vert^{p-2} \Vert \Re \Vert_{H^2}^2.
\end{equation}
Similarly as in the above proofs let $\tau_a$ and $\tau_b$ be stopping times such that
\[
	0 \leq \tau_a \leq \tau_b \leq t \wedge \tau^{T, \varepsilon}_K \wedge \tau^{z, \varepsilon}_K \wedge \tau^{\grad, \varepsilon}_K \wedge \tau^{6, \varepsilon}_K \wedge \tau^{w, \varepsilon}_K \wedge \tau^0_K.
\]
Using the Burkholder-Davis-Gundy inequality \eqref{eq:bdg} and similar estimates as in \eqref{eq:w.i1} and \eqref{eq:w.i2} we deduce
\begin{equation}
	\label{eq:mdp.R.i6}
	\begin{split}
	p \lambda^{-1}(\varepsilon) &\Eb \sup_{s \in [\tau_a, \tau_b]} \left| \int_{\tau_a}^s I_6 \, dW \right|\\
		&\leq \left(1 - \delta + \eta\right) \Eb \sup_{s \in [\tau_a, \tau_b]} \Vert \Re \Vert^p + C_\eta \lambda^{-2}(\varepsilon) \Eb \int_{\tau_a}^{\tau_b} \Vert \Re \Vert^{p-2} \left( 1 + \Vert U^0 \Vert_{H^2}^2 \right) \, ds\\
		&\hphantom{\leq \ } + C_\eta \varepsilon \Eb \int_{\tau_a}^{\tau_b} \Vert \Re \Vert^{p} \, ds + \frac{C_\eta \eta_3 \varepsilon}{1 - \delta} \Eb \int_{\tau_a}^{\tau_b} \Vert \Re \Vert^{p-2} \Vert \Re \Vert_{H^2}^2 \, ds.
	\end{split}
\end{equation}
Collecting the above estimates, choosing $\eta, \delta > 0$ and $\varepsilon_0$ sufficiently small we obtain
\[
	\Eb \left[ \sup_{s \in [\tau_a, \tau_b]} \Vert \Re \Vert^p + \int_{\tau_a}^{\tau_b} \Vert \Re \Vert^{p-2} \Vert \Re \Vert^2_{H_2} \, ds \right] \leq C_p \Eb \int_{\tau_a}^{\tau_b}\left( 1 + \Vert \Re \Vert^p \right) \left( 1 + \Phi(s) \right) \, ds,
\]
where the constant $C_p$ depends on $p$ but is independent of $\tau_a$ and $\tau_b$ and
\[
	\Phi(s) = 1 +  \vert \upe \vert_{L^6}^4 + \vert \partial_z \Re \vert^2 \Vert \partial_z \Re \Vert^2 + \Vert U^0 \Vert^2 \Vert U^0 \Vert_{H^2}^2 + \Vert U^0 \Vert_{H^2}^2.
\]

The proof is closed using the Gronwall lemma from Proposition \ref{prop:stoch.Gronwall} similarly as in the proofs above.
\end{proof}

\subsection{Proof of Theorem \ref{thm:mdp}}
\label{sect:mdp.proof}

The proof of Theorem \ref{thm:mdp} follows the argument of the proof of Theorem \ref{thm:ldp} in Section \ref{sect:ldp.proof} using the results of Sections \ref{sect:mdp.skeleton} and \ref{sect:mdp.preliminary} and is therefore omitted. The respective good rate function $I: C([0, t], V) \cap L^2(0, t; D(A)) \to \Rb$ is given by
\begin{equation}
	\label{eq:mdp.rate}
	I(U) = \inf \left\lbrace \frac12 \int_0^t \vert h \vert_{\Uc}^2 \, ds \mid h \in L^2(0, t; \Uc) \ \text{s.t.} \ U = \Gc^0_R\left(\int_0^\cdot h \, ds\right) \right\rbrace
\end{equation}
where $\Gc^0_R$ has been defined in Proposition \ref{prop:mdp.skeleton}. \qed

\section{Proof of Theorem \ref{thm:clt}}
\label{sect:clt}

Most of the estimates in this section are a straightforward adaptation of the estimates from Section \ref{sect:mdp.preliminary}. Thus we only go through the main steps of the proof which closely follows the one from \cite{zhang2019}.

Let $U^\varepsilon$ and $U^0$ be the solutions of \eqref{eq:ldp.main} and \eqref{eq:mdp.u.zero}, respectively, and recall that for $K \geq 0$ the stopping time
\[
	\tau_K^0 = \inf \left\lbrace s \geq 0 \mid \int_0^s \Vert U^0 \Vert_{H^2}^2 + \Vert U^0 \Vert^2 \Vert U^0 \Vert^2_{H^2} \, dr \right\rbrace
\]
satisfies $\tau_K^0 \to \infty$ as $K \to \infty$. Let $\Re = (U^\varepsilon - U^0)/\sqrt{\varepsilon}$. Clearly, $\Re$ satisfies the equation
\begin{equation}
	\label{eq:r.varep.clt}
	d\Re + \left[ A\Re + B\left( \Re, U^0 + \sqrt{\varepsilon} \Re \right) + B\left(U^0, \Re \right) + \Apr \Re + E \Re \right] \, dt = \sigma\left( \Ue \right) \, dW,
\end{equation}
with $\Re(0) = 0$, which is essentially \eqref{eq:r.varep} with $\lambda(\varepsilon) \equiv 1$. Let $\tau_K^{R, \varepsilon}$ be the stopping time defined for $K \geq 0$ by
\[
	\tau_K^{R, \varepsilon} = \inf \left\lbrace s \geq 0 \mid \sup_{r \in [0, s]} \Vert \Re \Vert^2 + \int_0^s \Vert \Re \Vert_{H^2}^2 + \Vert \Re \Vert^2 \Vert \Re \Vert_{H^2}^2 \, dr \right\rbrace.
\]
By the result of Section \ref{sect:mdp.preliminary} we have $\tau_K^{R, \varepsilon} \to \infty$ $\Pb$-a.s.\ as $K \to \infty$. Let $\Uh$ be the solution of the equation
\begin{equation}
	\label{eq:clt.uh}
	d\Uh + \left[ A\Uh + B(U^0, \Uh ) + B( \Uh, U^0 ) + \Apr \Uh  + E\Uh \right] \ ds = \sigma( U^0 ) \, dW, \quad \Uh(0) = 0.
\end{equation}
Let $\Ye = \Re - \Uh$. Using bilinearity of $B$ we observe that $\Ye$ satisfies the equation
\begin{multline}
	\label{eq:clt.ye}
	d\Ye + \left[ A\Ye + \sqrt{\varepsilon} B(\Re, \Re) + B(U^0, \Ye) + B(\Ye, U^0) + F_d(\Ye) \right] \, dt\\
	= \left[ \sigma(\Ue) - \sigma(U^0) \right] \, dW, \quad \Ye(0) = 0.
\end{multline}
Let $\tau_a$ and $\tau_b$ be stopping times such that $0 \leq \tau_a \leq \tau_b \leq t \wedge \tau_K^{R, \varepsilon} \wedge \tau_K^0$. By the It\^{o} lemma  and similar estimates as in Proposition \ref{prop:mdp.R} with the estimate \eqref{eq:b.estimate2} on $B$, the Lipschitz continuity \eqref{eq:sigma.lip.V} of $\sigma$ in $L_2(\Uc, V)$ and the boundedness of the operators $\Apr$ and $E$ we obtain
\begin{multline*}
	\Eb \left[ \sup_{s \in [\tau_a, \tau_b]} \Vert \Ye \Vert^2 + \int_{\tau_a}^{\tau_b} \Vert \Ye \Vert_{H^2}^2 \, ds \right] \leq C \Eb \int_{\tau_a}^{\tau_a} \Vert \Ye \Vert^2 \left( 1 + \Vert U^0 \Vert^2 \Vert U^0 \Vert_{H^2}^2 \right) \, ds\\
	+ C \Eb \Vert \Ye(\tau_a) \Vert^2 + C \varepsilon \Eb \int_{\tau_a}^{\tau_b} \Vert \Re \Vert^2 + \Vert \Re \Vert_{H^2}^2 + \Vert \Re \Vert^2 \Vert \Re \Vert_{H^2}^2 \, ds.
\end{multline*}
By the stochastic Gronwall lemma from \cite[Lemma 5.3]{glatt-holtz2009} we have
\begin{multline}
	\label{eq:clt.convergence}
	\Eb \left[ \sup_{s \in [0, t \wedge \tau_K^{R, \varepsilon}]} \Vert \Ye \Vert^2 + \int_0^{t \wedge \tau_K^{R, \varepsilon}} \Vert \Ye \Vert_{H^2}^2 \, ds \right]\\
	\leq C_{K} \varepsilon \Eb \int_{0}^{t \wedge \tau_K^{R, \varepsilon}} \Vert \Re \Vert^2 + \Vert \Re \Vert_{H^2}^2 + \Vert \Re \Vert^2 \Vert \Re \Vert_{H^2}^2 \, ds.
\end{multline}
If we define
\[
	\Omega_K = \left\lbrace \omega \in \Omega \mid \tau_K^0 \wedge \tau_K^{R, \varepsilon} \geq t \right\rbrace
\]
and consider the process $\mathds{1}_{\Omega_K} \Ye$ instead of $\Ye$, keeping $K$ fixed in \eqref{eq:clt.convergence} and taking the limit $\varepsilon \to 0+$, we get $\Ye \to 0$ in $C([0, t], V ) \cap L^2(0, t; H^2)$ $\Pb$-a.s.\ on $\Omega_K$. Since $\tau_K^0 \wedge \tau_K^{R, \varepsilon} \to \infty$ as $K \to \infty$ $\Pb$-almost surely, we have $\Pb(\Omega \setminus \cup_{K \in \Nb} \Omega_K) = 0$, which gives the convergence $\Pb$-a.s.\ on $\Omega$. \qed

\section*{Acknowledgement}

The author wishes to thank prof.\ Z.\ Brze\'{z}niak and P.\ Razafimandimby for fruitful discussions and the University of York for their kind hospitality.

\appendix

\section{Uniform version of stochastic Gronwall lemma}

The following result is not new, in fact it is a combination of the stochastic Gronwall lemma from \cite[Lemma 5.3]{glatt-holtz2009} and a part of the proof of global existence of the strong solutions of 2D stochastic Navier-Stokes equation from \cite[Theorem 4.2]{glatt-holtz2009}. We include the proof for the sake of completeness.

\begin{proposition}
\label{prop:stoch.Gronwall}
Let $\varepsilon_0 > 0$. Let $X^\varepsilon, Y^\varepsilon, Z^\varepsilon, R^\varepsilon: [0, \infty) \times \Omega \to [0, \infty)$ be stochastic processes on a probability space $(\Omega, \Fc, \Pb)$. Let $\tau_K^{R, \varepsilon}$ be the stopping time defined for $K > 0$ and $\varepsilon \in (0, \varepsilon_0]$ by
\[
	\tau_K^{R, \varepsilon} = \inf \left\lbrace t \geq 0 \mid \int_0^t R^\varepsilon \, ds \geq K \right\rbrace
\]
and let for all $t > 0$
\begin{equation}
	\label{eq:gronwall.R.convergence}
	\lim_{K \to \infty} \Pb \left( \left\lbrace \tau_K^{R, \varepsilon} \leq t \right\rbrace \right) = 0 \ \text{\emph{uniformly} w.r.t.\ } \varepsilon \in (0, \varepsilon_0].
\end{equation}
Let $T > 0$ and let for all $K > 0$ and $\varepsilon \in (0, \varepsilon_0]$
\[
	\Eb \int_0^{T \wedge \tau_K^{R, \varepsilon}} R^\varepsilon X^\varepsilon + Z^\varepsilon \leq C_{T, K} < \infty.
\]
Let there exist a constant $C_0 = C_0(T)$ such that for all $\varepsilon \in (0, \varepsilon_0]$ and all stopping times $\tau_a$ and $\tau_b$ satisfying $0 \leq \tau_a \leq \tau_b \leq T \wedge \tau_K^{R, \varepsilon}$ one has
\[
	\Eb \left[ \sup_{s \in [\tau_a, \tau_b]} X^\varepsilon + \int_{\tau_a}^{\tau_b} Y^\varepsilon \, ds \right] \leq C_0 \left[ X(\tau_a) + \int_{\tau_a}^{\tau_b} R^\varepsilon X^\varepsilon + Z^\varepsilon \, ds \right].
\]
Then for all $\varepsilon \in (0, \varepsilon_0]$ and $K > 0$ we have
\begin{equation}
	\label{eq:gronwall.inequality}
	\Eb \left[ \sup_{s \in \left[0, T \wedge \tau_K^{R, \varepsilon}\right]} X^\varepsilon + \int_0^{T \wedge \tau_K^{R, \varepsilon}} Y^\varepsilon \, ds \right] \leq C_{C_0, T, K} \Eb \left[ X(0) + \int_0^{T \wedge \tau_K^{R, \varepsilon}} Z^{\varepsilon} \, ds \right]
\end{equation}
and if we define the stopping time $\tau_K^{X, \varepsilon}$ for $K > 0$ and $\varepsilon \in (0, \varepsilon_0]$ by
\[
	\tau_K^{X, \varepsilon} = \inf \left\lbrace t \geq 0 \mid \sup_{s \in [0, t]} X^\varepsilon + \int_0^t Y^\varepsilon \, ds \geq K \right\rbrace,
\]
for all $t > 0$
\begin{equation}
	\label{eq:gronwall.convergence}
	\lim_{K \to \infty} \Pb \left( \left\lbrace \tau_K^{X, \varepsilon} \leq t \right\rbrace \right) = 0 \ \text{\emph{uniformly} w.r.t.\ } \varepsilon \in (0, \varepsilon_0].
\end{equation}
\end{proposition}

\begin{proof}
The inequality \eqref{eq:gronwall.inequality} is the stochastic Gronwall lemma from \cite[Lemma 5.3]{glatt-holtz2009}. Following the argument from \cite[Theorem 4.2]{glatt-holtz2009} we use the Chebyshev theorem and \eqref{eq:gronwall.inequality} to estimate
\begin{align*}
	\Pb\left( \left\lbrace \tau_K^{X, \varepsilon} \leq t \right\rbrace \right) &\leq \Pb \left(  \left\lbrace \tau_K^{X, \varepsilon} \leq t \right\rbrace \cap \left\lbrace \tau_M^{R, \varepsilon} > t \right\rbrace \right) + \Pb \left( \left\lbrace \tau_M^{R, \varepsilon} \leq t \right\rbrace \right)\\
	&\leq \Pb \left( \left\lbrace \sup_{s \in \left[0, t \wedge \tau_M^{R, \varepsilon} \right]} X^\varepsilon + \int_0^{t \wedge \tau_M^{R, \varepsilon}} Y^\varepsilon \, ds \geq K \right\rbrace \right) + \Pb \left( \left\lbrace \tau_M^{R, \varepsilon} \leq t \right\rbrace \right)\\
	&\leq \frac{1}{K} \Eb\left[ \sup_{s \in \left[0, t \wedge \tau_M^{R, \varepsilon} \right]} X^\varepsilon + \int_0^{t \wedge \tau_M^{R, \varepsilon}} Y^\varepsilon \, ds \right] + \Pb \left( \left\lbrace \tau_M^{R, \varepsilon} \leq t \right\rbrace \right)\\
	&\leq \frac{C_{t, M}}{K} + \Pb \left( \left\lbrace \tau_M^{R, \varepsilon} \leq t \right\rbrace \right). 
\end{align*}
Let $\delta > 0$ be arbitrary. By the uniform convergence \eqref{eq:gronwall.R.convergence} we find $M \in \Nb$ such that for all $\varepsilon \in (0, \varepsilon_0]$ we have
\[
	\Pb \left( \left\lbrace \tau_M^{R, \varepsilon} \leq t \right\rbrace \right) < \frac{\delta}{2}.
\]
Let $K_0 \in \Nb$ be such that for all $K \in \Nb$, $K \geq K_0$, we have $C_{t, M}/K < \delta/2$. Collecting the above we have that for all $\varepsilon \in (0, \varepsilon_0]$ and all $K \in \Nb$, $K \geq 0$
\[
	\Pb\left( \left\lbrace \tau_K^{X, \varepsilon} \leq t \right\rbrace \right) < \delta,
\]
which finishes the proof of the uniform convergence \eqref{eq:gronwall.convergence}.
\end{proof}

\addcontentsline{toc}{section}{References}
\bibliography{bibliography}
\bibliographystyle{plain}
\end{document}